\documentclass[11pt]{amsart}
 \usepackage{amsmath}
 \usepackage{hyperref}
\usepackage{latexsym,amssymb,amsfonts}
 \usepackage[totalwidth=500pt, totalheight=660pt]{geometry}

\usepackage{enumerate}
\usepackage{amsmath}
\usepackage{mathbbol}
\usepackage[latin1]{inputenc}
\numberwithin{equation}{section}
\newtheorem{theo}{Theorem}

\newtheorem{lemma}[theo]{Lemma}
\newtheorem{prop}[theo]{Proposition}
\newtheorem{cor}[theo]{Corollary}
\newtheorem{defi}[theo]{Definition}

\theoremstyle{definition}

\newtheorem{rem}[theo]{Remark}
\newtheorem{rems}[theo]{Remarks}

\numberwithin{theo}{section}
\usepackage{eufrak}

\def\me{\mathsf{e}}
\def\mv{\mathsf{v}}
\def\ea{\EuFrak{a}}

\def\@secappcntformat#1{%
 \ifappendix \rm\appendixname\ifoneappendix\else~\fi\fi
 \ifoneappendix \else \csname the#1\endcsname\relax\fi
 \ifx\apphe@d\@empty \else .\fi\enskip
}
\title[Gaussian estimates for a heat equation]{Gaussian estimates for a heat equation\\ on a network}
\author[Delio Mugnolo]{}

\email{delio.mugnolo@uni-ulm.de}
\subjclass{Primary: 47D06, 47A07; Secondary 90B10}
\keywords{Evolution equations on networks; Ultracontractive semigroups of operators; Gaussian estimates}
\thanks{I gratefully acknowledge the hospitality of the AGFA - Working group on Functional Analysis at the Eberhard-Karl University of T\"ubingen, where this paper was begun in February 2005. I would also like to thank Rainer Nagel, Susanna Piazzera, Stefan Teufel and Roderich Tumulka (University of T\"ubingen), as well as Abdelaziz Rhandi (University of Marrakech) and Stefano Cardanobile (University of Ulm) for helpful comments and motivating discussions.}

\begin{document}
\maketitle

\centerline{\scshape Delio Mugnolo}
\medskip
{\footnotesize
 \centerline{Abteilung Angewandte Analysis der Universit\"at}
  \centerline{Helmholtzstra{\ss}e 18, D-89081 Ulm, Germany}
\smallskip
\centerline{and}
\smallskip
 \centerline{Dipartimento di Matematica dell'Universit\`a}
  \centerline{Via Orabona 4, I-70125 Bari, Italy}
} 

\medskip


 \medskip

\begin{abstract}
We consider a diffusion problem on a network on whose nodes we
impose Dirichlet and generalized, non-local Kirchhoff-type
conditions. We prove well-posedness of the associated initial
value problem, and we exploit the theory of sub-Markovian and
ultracontractive semigroups in order to obtain upper Gaussian
estimates for the integral kernel. We conclude that the same
diffusion problem is governed by an analytic semigroup acting on
all  $L^p$-type spaces as well as on suitable spaces of continuous
functions. Stability and spectral issues are also discussed. As an
application we discuss a system of semilinear equations on a
network related to potential transmission problems arising in
neurobiology.
\end{abstract}

\maketitle

\section{Introduction}

Evolution equations taking place in networks or, more generally, in ramified structures have been first considered in pioneering articles by K. Ruedenberg and C. Scherr back in the 1950s, cf. \cite{RS53}, and, at a more mathematical level, in a series of papers by R. Mills and E. Montroll and by G. Lumer in the 1970s, cf. \cite{Mo70}--\cite{MM70} and~\cite{Lu80}--\cite{Lu80b}, respectively. Shortly afterwards, F. Ali Mehmeti, J. von Below, S. Nicaise, and J.P. Roth among others began a systematical study of properties of elliptic operators acting on spaces of functions over networks, cf. e.g. the monographs~\cite{Ni93},~\cite{AM94},~\cite{LLS94}, and references therein. Ever since, such problems have aroused broad interest among mathematicians working on partial differential equations, control, and spectral theory -- as well as among theoretical physicists interested in scattering theory of guided waves, photonic crystals, and quantum wires, resulting in a literature so vast that it can by no means be summarized here.

Throughout this paper we consider a finite, unitarily
parametrized, connected network whose structure is given by a
suitable graph. On it we study a general diffusion equation.
Adopting a setting which is standard in literature, the node
conditions impose continuity and Kirchhoff-type transmission laws
in ramification vertices. However, we extend known results by
allowing more general, \emph{non-local} node conditions,
cf.~Section 2 below. Roughly speaking, in each node $\mv_i$ of the
network we allow an absorption that does not only depend on the
value of the function in $\mv_i$ itself, but also on other nodes
$\mv_h$. Using the arguments presented in~\cite{Be88}, one can
promptly obtain well-posedness of such a general parabolic network
problem; we are thus more interested in qualitative properties. In
some sense, our results complement those obtained in~\cite{KS99}
and~\cite{KS06}, where even more general node conditions are
allowed, and where positivity of the semigroups is also discussed,
but where the quantum physical viewpoint was motivated to mainly
discuss those conditions leading to self-adjointness.


In this paper we pursue an approach to parabolic equations on
networks based on the theory of sesquilinear forms and associated
sub-Markovian semigroups. Following the approach presented in
previous works (see~\cite{Be88}, ~\cite[Chapt.~2]{AM94},
and~\cite{KMS06}), we draw some conclusions about several issues,
including $L^\infty$-contractivity of the semigroup governing the
problem, its $L^2\!-\!L^\infty$-stability and its analyticity in
suitable spaces of continuous as well as of $L^1$ functions over
the graph. The key arguments are upper estimates for the integral
kernel of the generated semigroup (which, roughly speaking, yields
the Green function of the problem). The theory of such
\emph{Gaussian estimates} has become a mature one that proves
extremely powerful when applied to diffusion problems on domains
-- and, as a matter of fact, on networks, too; we refer
to~\cite{Da89},~\cite{Ar97}, and~\cite{Ou04} for an introduction
to this subject.

While the integral kernel for the diffusion problem on a network has already been explicitly computed in~\cite{Ni87} 
in the special case of constant coefficients for the heat
operator, Gaussian estimates are, to our knowledge, new in the
case of \emph{variable} coefficients, even in the case of local
nodal conditions. Observe also that Gaussian estimates for the
semigroup that governs the \emph{discrete} diffusion problem on a
graph have recently been proved in~\cite{Km05}, cf.
also~\cite{BCK05}.

\section{General framework}

We consider a finite connected network, represented by a finite graph $G$ with $m$ edges $\me _1,\dots,\me _m$ and $n$ vertices $\mv_{1},\dots,\mv_{n}$. We normalize and parametrize the edges on the interval $[0,1]$. The structure of the network is given by the $n\times m$ matrix
$\Phi^+:=(\phi^+_{ij})$ and $\Phi^-:=(\phi^-_{ij})$ defined by
\begin{equation*}\label{incid}
\phi^+_{ij}:=
\left\{
\begin{array}{rl}
1, & \hbox{if } \me _j(0)=\mv _i,\\
0, & \hbox{otherwise},
\end{array}
\right.
\quad\hbox{and}\quad
\phi^-_{ij}:=
\left\{
\begin{array}{rl}
1, & \hbox{if } \me _j(1)=\mv _i,\\
0, & \hbox{otherwise}.
\end{array}
\right.
\end{equation*}
The $n\times m$ matrix $\Phi:=(\phi_{ij})$, defined by $\Phi:=\Phi^+-\Phi^-$, is thus the
\emph{incidence matrix} of the graph $G$. Further, let $\Gamma(\mv_i)$ be the set of all the indices of the edges having an endpoint at $\mv _i$, i.e.,
$$\Gamma(\mv _i):=\left\{j\in \{1,\ldots,m\}: \me _j(0)=\mv _i\hbox{ or } \me _j(1)=\mv _i\right\}, \qquad 1\leq i\leq n-1.$$
For the sake of simplicity, we denote the value of the functions $c_j(\cdot)$ and ${u}_j(t,\cdot)$ at $0$ or $1$ by $c_j(\mv_i)$ and $u_j(t,\mv_i)$, if $\me_j(0)=\mv _i\hbox{ or
} \me _j(1)=\mv _i$, respectively. With an abuse of notation, we also set $u'_j(t,\mv_i)=c_j(\mv_i):=0$ whenever $j\notin\Gamma(\mv_i)$.

In the literature networks are usually considered where Dirichlet
conditions are imposed on $n_0$ boundary vertices (i.e., vertices of
degree 1). Since no process takes place in such boundary vertices,
we may and do identify all of them. Thus, we instead consider an
equivalent diffusion equation on a rearranged graph where these
$n_0$ nodes of degree 1 are replaced by \emph{only one node of
degree $n_0$}, on which a Dirichlet condition is imposed. Without
loss of generality we assume that this node is $\mv_{n}$.

For $t\geq 0$ we are now in the position to consider the network
diffusion problem
\begin{equation*}
\left\{\begin{array}{rcll}
\dot{u}_j(t,x)&=& (c_j u_j')'(t,x), &x\in(0,1), \; j=1,\dots,m,\\
u_j(t,\mv _i)&=&u_\ell (t,\mv _i)=: d^u_i(t), &j,\ell\in \Gamma(\mv _i),\; i=1,\ldots,n,\\
\sum_{h=1}^{n-1} b_{ih} d^u_h(t)&=& \sum_{j=1}^m \phi_{ij} c_j(\mv_i) u'_j(t,\mv _i), & i=1,\ldots,n-1,\\
d^u_n(t)&=&0,
\end{array}
\right.
\end{equation*}
on the network, with initial conditions
$$u_j(0,x)=u_{0j}(x), \qquad    x\in (0,1), \; j=1,\dots,m.$$
Throughout this paper we assume that the coefficients $c_j$
satisfy $0<c_j\in C^1[0,1]$, $j=1,\ldots,m$, while for the time
being the $(n-1)\times (n-1)$ matrix $B:=(b_{ih})_{1\leq i,h\leq
n-1}$ is arbitrary, i.e., $B\in M_{n-1}({\mathbb C})$. The third
equation above is a generalized, non-local Kirchhoff-type law on
the first $n-1$ nodes, while the fourth one prescribes a Dirichlet
condition in $\mv_{n}$.


We introduce the $n\times m$ matrices $\Phi^+_w:=(\omega^+_{ij})$ and $\Phi^-_w:=(\omega^-_{ij})$ defined by
\begin{equation*}
\omega^+_{ij}:=
\left\{
\begin{array}{ll}
c_j(\mv_i), & \hbox{if } \phi^+_{ij}=1\hbox{ and }i\leq n-1,\\
0, & \hbox{otherwise},
\end{array}
\right.
\end{equation*}
and
\begin{equation*}
\omega^-_{ij}:=
\left\{
\begin{array}{ll}
c_j(\mv_i), & \hbox{if } \phi^-_{ij}=1\hbox{ and }i\leq n-1,\\
0, & \hbox{otherwise}.
\end{array}
\right.
\end{equation*}

In the following we will repeatedly and without further notice write the functions $u_j$ in vector form, i.e., $$u\equiv\begin{pmatrix}u_1\\ \vdots\\ u_m\end{pmatrix}.$$
With these notations, one can directly check that the second, third and fourth equations of our network diffusion problem can be rewritten as
\begin{equation*}\label{cont}
\exists d^u(t)\in {\mathbb C}^{n-1}\times \{0\} \hbox{ s.t. }
\hspace{-10pt}
\begin{array}{rl}
&(\Phi^+)^\top d^u(t)=u(t,0),\; (\Phi^-)^\top d^u(t)=u(t,1),\\
&\hbox{and }\Phi_w^- u'(t,1)-\Phi_w^+ u'(t,0)={\mathbb B}d^u(t),
\end{array}
\quad t\geq 0,
\end{equation*}
where we have introduced the $n\times n$ matrices
$${\mathbb B}:=\begin{pmatrix} B & 0\\ 0 & 0\end{pmatrix}.$$


\begin{rem}
Inspired by the non-local boundary conditions introduced above, and motivated by the theory of ephaptic coupling of myelinated nerve fibers (see~\cite[Chapt.~8]{Sc02}) one may also consider an ever more general parabolic system over the network, where for $t\geq 0$ the first equation of the problem is replaced by
$$\dot{u}_j(t,x)= \sum_{\ell=1}^m (c_{j\ell} u_\ell')'(t,x),\qquad x\in(0,1), \; j=1,\dots,m.$$
Under suitable assumptions on the diffusion coefficients $c_{ij}$
one can study well-posedness and qualitative properties of such a
system. However, this introduces new technical difficulties, as
for instance one sees that usual Kirchhoff- conditions do not
ensure dissipativity of such a system any more, and one has to
formulate appropriate, more general conditions in the nodes. We will discuss this kind of
problems in the forthcoming paper~\cite{CM06b}.
\end{rem}

It is already well-known that the above network diffusion problem is well-posed: in fact, this has been shown in a Hilbert space setting in~\cite{Be88} for the case of variable coefficients and diagonal $B$, cf. also the references therein for earlier results on less general cases. We also remark that, at least in the case of $c_1=\ldots=c_m\equiv1$ and $b_{ih}=0$, $i,h=1,\ldots,n$. S. Nicaise has derived in~\cite{Ni87} an explicit formula for the solution, thus showing the well-posedness of the problem in all $L^p$-type spaces as well as on suitable spaces of continuous functions (see also~\cite{Ni86}). Our first goal is to to establish a meaningful $L^p$-theory for the general case of variable diffusion coefficients and non-local node conditions. To this aim, we need the following.

\begin{defi}\label{xplp}
For given functions $f_j:[0,1]\to{\mathbb C}$, $j=1,\ldots,n$, we define a mapping $Uf:[0,m]\to {\mathbb C}$ by
\begin{equation*}
Uf(x):=\tilde{f}(x):=f_j(x-j+1)\qquad \hbox{if }x\in (j-1,j),\; j=1,\ldots,m.
\end{equation*}
\end{defi}

With this notation one sees that the following holds.

\begin{lemma}
The mapping $U$ is one-to-one from $X_p:=\left(L^p(0,1)\right)^m$ onto $L^p(0,m)$ for all $p\in [1,\infty]$, and in fact it is an isometry if we endow $\left(L^p(0,1)\right)^m$ with the canonical $l^p$-norm, i.e.,
$$\Vert f\Vert_{X_p}:=\left(\sum_{j=1}^m \Vert f_j\Vert_{L^p(0,1)}^p\right)^\frac{1}{p},\qquad 1\leq p<\infty,$$
or
$$\Vert f\Vert_{X_\infty}:=\max_{1\leq j\leq m} \Vert f_j\Vert_{L^\infty(0,1)}.$$
\end{lemma}

In the following we will
hence regard $X_p$ as an $L^p$-space over a finite measure space, so that $X_p\hookrightarrow X_q$ for all $1\leq q\leq p\leq\infty$. Moreover, each $X_p$ is a Banach lattice, and its positive cone can be identified with the positive cone of $L^p(0,m)$.

\section{Basic results}

We are now in the position to consider an abstract reformulation of our diffusion problem. First we consider the (complex) Hilbert space $X_2=\left(L^2(0,1)\right)^m$ endowed with the natural inner product
$$(f\mid g)_{X_2}:= \sum_{j=1}^m \int_0^1 f_j(x)\overline{g_j(x)} dx,\qquad
f,g\in X_2.$$
On $X_2$ we define an operator
\begin{equation}\label{amain}
A:=\begin{pmatrix}
\frac{d}{dx}\left(c_1 \frac{d}{dx}\right) & & 0\\
 & \ddots &\\
0 & & \frac{d}{dx}\left(c_m \frac{d}{dx}\right) \\
\end{pmatrix}
\end{equation}
with domain
\begin{equation}\label{domamain}
D(A):=\left\{f\in \left(H^2(0,1)\right)^m:
\hspace{-10pt}
\begin{array}{rl}
& \exists d^f\in {\mathbb C}^{n-1}\times \{0\} \hbox{ \rm{s.t.}} \\
&(\Phi^+)^\top d^f=f(0),\; (\Phi^-)^\top d^f=f(1),\\
&\hbox{ and }\Phi_w^- f'(1)-\Phi_w^+ f'(0)={\mathbb B}d^f\\
\end{array}
\right\}.
\end{equation}

We can finally rewrite the concrete network diffusion problem as an abstract Cauchy problem
\begin{equation}\tag{ACP}
\left\{\begin{array}{rcll}
\dot{u}(t)&=& Au(t),\qquad   &t\geq 0,\\
u(0)&=&u_{0},\\
\end{array}
\right.
\end{equation}
on $X_2$. In order to discuss the semigroup generator property of $A$ it is convenient to use a variational method.

Recall that we are assuming the network to be connected throughout the paper. This is crucial for the proof of the following.

\begin{lemma}\label{v0}
The linear space
\begin{equation*}\label{domformmain}
V_0:=\left\{f\in \left(H^1(0,1)\right)^m:\!\!\!\!\!\!\!
\begin{array}{ll}
&\exists d^f\in {\mathbb C}^{n-1}\times \{0\} \hbox{ s.t. }\\
&(\Phi^+)^\top d^f=f(0)\hbox{ and } (\Phi^-)^\top d^f=f(1)
\end{array}\right\}.
\end{equation*}
is densely and compactly embedded in ${\mathcal X}_2$. It becomes a Hilbert space when equipped with the inner product
\begin{equation}\label{normeq}
(f\mid g)_{V_0}:=\sum_{j=1}^m\int_0^1 f'_j(x)\overline{g'_j(x)} dx,\qquad f,g\in V_0.
\end{equation}
\end{lemma}

\begin{proof}
It is well-known that $V_0$ is  densely and compactly imbedded in $X_2$: this follows from the inclusions $\left(C^\infty_c(0,1)\right)^m\subset V_0\subset (H^1(0,1))^m$ and the Rellich--Khondrakov Theorem.

We are going to show that the inner product defined in~\eqref{normeq} is equivalent to that induced by the Hilbert space $(H^1(0,1))^m$, i.e., to
$$(f\mid g):=\sum_{j=1}^m\int_0^1 \left(f'_j(x)\overline{g'_j(x)}+f_j(x)\overline{g_j(x)}\right) dx,\qquad f,g\in (H^1(0,1))^m.$$

Let $\me_j$ be a general edge. By the connectedness of $G$ we can find a set of edges $\me_{j_1},\ldots,\me_{j_r}$ linking any $x\in\me_j$ with $\mv_{n}$. More precisely, there exist ${j_1},\ldots,{j_r}\in\{1,\ldots,m\}$ such that
\begin{itemize}
\item $\phi_{nj_1}^- = 1$, i.e., $\mv_{n}$ is the head of $\me_{j_1}$;
\item for all $h=1,\ldots,r$ there exists a vertex $\mv_{i_{h+1}}$ such that $\phi_{i_{h+1}j_h}^+=1=\phi_{i_{h+1}j_{h+1}}^-$, i.e, $\mv_{i_{h+1}}$ is the tail of $\me_{j_h}$ and the head of $\me_{j_{h+1}}$;
\item $\me_{j_r}=\me_j$.
\end{itemize}
Then, for every $f\in V$ and at any $x\in(0,1)$ one has $f_{j_r}(x)= d^f_{i_{r-1}}+ \int_0^x     f'_{j_r}(s) ds$, and therefore
\begin{eqnarray*}
\vert f_{j_r}(x)\vert &\leq &\vert d^f_{i_{r-1}}\vert + \int_0^1 \vert f'_{j_r}(s)\vert ds\\
& \leq & \vert d^f_{i_{r-2}}\vert + \int_0^1 \vert f'_{j_{r-1}}(s)\vert ds +\int_0^1 \vert f'_{j_r}(s)\vert ds\\
&\leq& \ldots\\
& \leq & \vert d^f_{n}\vert + \sum_{h=1}^r \int_0^1 \vert f'_{j_{h}}(s)\vert ds\\
&=& \sum_{h=1}^r \int_0^1 \vert f'_{j_{h}}(s)\vert ds,
\end{eqnarray*}
due to the Dirichlet condition satisfied at $\mv_{n}$ by all $f\in V$. We conclude that
$$\sum_{j=1}^m \int_0^1 \vert f_j(x)\vert dx \leq m \sum_{j=1}^m \int_0^1 \vert f_j'(x)\vert dx.$$
Having proved such a Poincar\'e-type inequality, the claim follows directly.
\end{proof}

The following lemma extends known results (see e.g.~\cite{AM86},~\cite{Ni87}, and~\cite{Be88}) to the case of non-local node conditions. Thus, for the sake of self-containedness we do not omit the proofs, although basic elements and techniques involved are essentially well-known.

\begin{prop}\label{main}
Consider the sesquilinear form
\begin{equation*}\label{formmain}
\ea(f,g):=\sum_{j=1}^m\int_0^1 c_j(x) f'_j(x) \overline{g'_j(x)} dx -\sum_{i,h=1}^{n-1} b_{ih} d^f_h\overline{d^g_i},\qquad f,g\in V_0,
\end{equation*}
on the Hilbert space $X_2$. Then $\ea$ enjoys the following properties:
\begin{itemize}
\item $\ea$ is $X_2$-elliptic, i.e.,: there exist $\omega\in{\mathbb R}$, $\alpha>0$ such that
$${\rm Re}\ea(f,f)+\omega\Vert f\Vert^2_{X_2}\geq \alpha \Vert f\Vert^2_{V_0}\qquad \hbox{for all }f\in V_0;$$
\item $\ea$ is continuous, i.e., there exists $M>0$ such that
$$\vert \ea(f,g)\vert \leq  M\Vert f\Vert_{V_0} \Vert g\Vert_{V_0}\qquad \hbox{ for all }f,g\in V_0;$$
\item $\ea$ is symmetric, i.e.,
 $$\ea(f,g)=\overline{\ea(g,f)}\qquad \hbox{for all }f,g \in V_0,$$ if and only if $B$ is self-adjoint;
\item $\ea$ is coercive, i.e., it is $X_2$-elliptic with $\omega=0$, if and only if it is accretive, i.e.,
$${\rm Re}\ea(f,f)\geq 0 \qquad \hbox{ for all }f \in V_0,$$
if and only if $B$ is dissipative.
\end{itemize}
\end{prop}


\begin{proof}
First of all, we show that $\mathfrak a$ is $X_2$-elliptic.   We observe that the leading term in the form $\mathfrak a$,
  i.e., 
  $${\mathfrak a}_0(f,g):=\sum_{j=1}^m\int_0^1 c_j(x) f'_j(x) \overline{g'_j(x)} dx,\qquad f,g\in V_0,$$ 
  is clearly $X_2$-elliptic and continuous. Furthermore, it follows from the Gagliardo--Nirenberg
  inequality
$$ \|u\|_\infty \le Const. \|u\|^\frac{1}{2}_{L^2}\|u\|^\frac{1}{2}_{H^1}\qquad \hbox{for all }u\in H^1(0,1),$$
see e.g.~\cite[Th\'eo.~VIII.7]{Bre83}, that the space 
$$C_0(G):=\left\{f\in \left(C[0,1]\right)^m: \!\! \begin{array}{ll}
&\exists d^f\in {\mathbb C}^{n-1}\times \{0\} \hbox{ s.t. }\\
&(\Phi^+)^\top d^f=f(0)\hbox{ and }(\Phi^-)^\top d^f=f(1)
                                                  \end{array}
 \right\},$$
 is embedded in an interpolation space of order $\frac{1}{2}$ between $V_0$ and $X_2$, and obviously $${\mathfrak a}_1(f,g):= -\sum_{i,h=1}^{n-1} b_{ih} d^f_h\overline{d^g_i}$$
 is bounded from $C_0({\mathsf G})\times C_0({\mathsf G})$ to ${\mathbb C}$. Thus~\cite[Lemma~2.1]{Mug08} applies and one concludes that also their sum ${\mathfrak a}={\mathfrak a}_0+{\mathfrak a}_1$ is $X_2$-elliptic.

Set  now
$$C:=\max_{1\leq j\leq m}\max_{x\in[0,1]}c_j(x)\quad\hbox{and}\quad b:=\max_{1\leq i,h\leq n-1} \vert b_{ih}\vert .$$

Recall that by~\eqref{normeq} there exists a constant $N$ such that $\sum_{j=1}^m\Vert f_j\Vert_{H^1(0,1)}\leq N\Vert f\Vert_{V_0}$. Take $f,g\in V_0$ and observe that
\begin{eqnarray*}
\vert \ea(f,g)\vert&\leq& C\left\vert \sum_{j=1}^m \int_0^1 f'_j(x) \overline{g'_j(x)} dx\right\vert+
\sum_{i,h=1}^{n-1} \vert b_{ih}\vert \vert d^f_i\vert \vert d^g_h\vert\\
&\leq& C\vert ( f\mid g)_{V_0}\vert+
b \sum_{i,h=1}^{n-1}  N^2 \Vert f\Vert_{V_0} \Vert g\Vert_{V_0}\\
&\leq & M \Vert f\Vert_{V_0} \Vert g\Vert_{V_0}
\end{eqnarray*}
by the Cauchy--Schwartz inequality, where $M:=C+bN^2(n-1)^2$. This completes the proof of the continuity of $\ea$.

Since by assumptions the coefficients $c_1,\ldots,c_m$ are strictly positive,
$$\ea(f,g)=\overline{\ea(g,f)} \qquad\hbox{if and only if}\qquad \sum_{i,h=1}^{n-1} b_{ih} d^f_h\overline{d^g_i}=\sum_{i,h=1}^{n-1} \overline{b_{ih}} d^f_i\overline{d^g_h}, \qquad f,g\in V_0,$$
and this is the case if and only if $B$ is self-adjoint.

Let finally $\ea$ be coercive. Then it is clearly accretive. If $\ea$ is accretive, then a direct computation shows that
$${\rm Re}\sum_{i,h=1}^{n-1} b_{ih} d^f_h\overline{d^f_i}\leq 0$$
holds for all $f\in V_0$. Due to the arbitrarity of the nodal values $d^f_2,\ldots, d^f_n$ of $f\in V_0$, one sees that this already implies that $B$ is dissipative. Finally, if $B$ is dissipative, then there holds
\begin{eqnarray*}
{\rm Re}{\ea}(f,f) &=& \sum_{j=1}^m\int_0^1 c_j(x) \vert f'_j(x)\vert^2 dx-{\rm Re}\sum_{i,h=1}^{n-1} b_{ih} d^f_h \overline{d^f_i}\\
&\geq& c \sum_{j=1}^m\int_0^1 \vert f'_j(x)\vert^2 dx,
\end{eqnarray*}
where 
$$c:=\frac{1}{2}\min_{1\leq j\leq m}\min_{x\in[0,1]} \Vert c_j(x)\Vert>0.$$
By Lemma~\ref{v0}, we have thus obtained that ${\rm Re}{\ea}(f,f)\geq \alpha \Vert f\Vert^2_{V_0}$ for some $\alpha>0$. This concludes the proof.
\end{proof}

\begin{cor}\label{sa+diss}
The operator associated with $\ea$ is densely defined, sectorial, and resolvent compact, hence it generates a strongly continuous, analytic, compact semigroup $(T_2(t))_{t\geq 0}$ on $X_2$. If moreover $B$ is dissipative (resp., self-adjoint), then $(T_2(t))_{t\geq 0}$ is contractive and uniformly exponentially stable (resp., self-adjoint).
\end{cor}

\begin{proof}
It follows from Proposition~\ref{main} and~\cite[Prop.~1.51, Thm.~1.52]{Ou04} that the operator associated with ${\ea}$ is densely defined, sectorial, and resolvent compact.

Let $B$ be dissipative and take $f$ in the null space of the operator associated with $\ea$. Then by the proof of the above lemma there exists $c>0$ such that
$$0={\rm Re}{\ea}(f,f)\geq c \sum_{j=1}^m\int_0^1  \vert f'_j(x)\vert^2 dx.$$
This means that $f_j$ is constant for all $j=1,\ldots, m$, hence $f$ is constant on the whole network, due to the continuity condition in the nodes. In particular, $f\equiv d^f_n	=0$, hence the operator associated with $\ea$ is one-to-one.
\end{proof}

In the following we are able to identify the operator associated with $\ea$. The following result is already known in the literature in the special case where the matrix $B$ is diagonal, see e.g.~\cite[Chapt.~2]{AM94}.

\begin{lemma}\label{ident}
The operator associated with the form $\ea$ is $\left(A,D(A)\right)$ defined in \eqref{amain}--\eqref{domamain}.
\end{lemma}

\begin{proof}
Denote by $\left(C,D(C)\right)$ the operator associated with $\ea$, which by definition is given by
$$\begin{array}{rcl}
D(C)&:=&\left\{f\in V_0:\exists g\in X_2 \hbox{ s.t. } \ea(f,h)=(g\mid u)_{X_2}\; \forall u\in V_0\right\},\\
Cf&:=&-g.
\end{array}$$
Let us first show that $A\subset C$. Take $f\in D(A)$. Then for all $u\in V_0$
\begin{eqnarray*}
\ea(f,u)&=&\sum_{j=1}^m \int_0^1 c_j(x) f'_j(x)\overline{u'_j(x)}dx -\sum_{i,h=1}^{n-1} b_{ih} d^f_h\overline{d^u_i}\\
&=& \sum_{j=1}^m\left[c_jf'_j\overline{u_j}\right]_0^1 - \sum_{j=1}^m \int_0^1 (c_j f_j')'(x)\overline{u_j(x)}dx  -\sum_{i,h=1}^{n-1} b_{ih} d^f_h\overline{d^u_i}.
\end{eqnarray*}
Using the incidence matrix $\Phi=\Phi^+ - \Phi^- $ and recalling that $d^u_n=0$ as $u\in V_0$, we can write
$$\sum_{j=1}^m\left[c_jf'_j\overline{u_j}\right]_0^1 = \sum_{j=1}^m\sum_{i,h=1}^{n-1} c_j(\mv_i)(\phi_{ij}^- -\phi_{ij}^+)f'_j(\mv _i)\overline{d^u_i}= \sum_{i,h=1}^{n-1} \overline{d^u_i}\sum_{j=1}^m (\omega_{ij}^- -\omega_{ij}^+)f'_j(\mv _i).$$
Using the generalized Kirchhoff condition $\Phi^-_w f'(1)-\Phi^+_w f'(0)={\mathbb B}d^f$, which holds for all functions $f\in D(A)$, we obtain that
\begin{eqnarray*}
\ea(f,u)&=& \sum_{i=1}^{n-1} \overline{d^u_i} \underbrace{\sum_{j=1}^m (\omega_{ij}^- -\omega_{ij}^+)f'_j(\mv _i)}_{=\sum_{h=1}^{n-1} b_{ih} d^f_h} - \sum_{j=1}^m \int_0^1 (c_jf_j')'(x)\overline{u_j(x)}dx\\
&&\qquad-\sum_{ i,h=1}^{n-1} b_{ih} d^f_h\overline{d^u_i}\\
&=&- \sum_{j=1}^m \int_0^1 (c_jf_j')'(x)\overline{u_j(x)}dx=-(Af\mid u)_{X_2},
\end{eqnarray*}
which makes sense because $Af\in X_2$. The proof of the inclusion $A\subset C$ is completed.

To check the converse inclusion take $f\in D(C)$. By definition, there exists $g\in X_2$ such that
\begin{equation}\label{integr}
\ea(f,u)=(g\mid u)_{X_2}= \sum_{j=1}^m \int_0^1 g_j(x)\overline{u_j(x)}dx\qquad \hbox{for all } u\in V_0,
\end{equation}
hence in particular for all functions of the form
\begin{equation}\label{formfun}
\begin{pmatrix}
0\\
\vdots\\
u_j\\
\vdots\\
0
\end{pmatrix}\leftarrow j^{\rm th}\hbox{ row},\;\; u_j\in H^1_0(0,1).
\end{equation}
From this follows that~\eqref{integr} in fact implies
$$\int_0^1 c_j(x)f'_j(x)\overline{u'_j(x)}dx = \int_0^1 g_j(x)\overline{u_j(x)}dx\hbox{\; for all } j=1,\ldots,m,\;\; u_j\in H^1_0(0,1).$$
By definition of weak derivative this means that $c_j\cdot f_j'\in H^1(0,1)$ for all $j=1,\ldots,m$. Since $0<c_j\in C^1[0,1]$, there follows that $f_j'\in H^1(0,1)$ for all $j=1,\ldots,m$. We conclude that $f\in \left(H^2(0,1)\right)^m$. Moreover, integrating by parts as in the proof of the first inclusion we see that if~\eqref{integr} holds for all $u\in V_0$, then there also holds
$$\sum_{i,h=1}^{n-1} \overline{d^u_i} \sum_{j=1}^m (\omega_{ij}^- -\omega_{ij}^+)f'_j(\mv _i)=
\sum_{i,h=1}^{n-1} b_{ih} d^f_h\overline{d^u_i}.$$
Since $u\in V_0$ is arbitrary, this means that
$$\sum_{j=1}^m (\omega_{ij}^- -\omega_{ij}^+)f'_j(\mv _i)=\sum_{h=1}^{n-1} b_{ih} d^f_h\qquad \hbox{for all }i=1,\ldots,n-1,$$
that is, $\Phi_w^- f'(1)-\Phi_w^+ f'(0)={\mathbb B}d^f$. Therefore $f\in D(A)$ and
\begin{equation*}
-\sum_{j=1}^m \int_0^1 (c_jf'_j)'(x)\overline{u_j(x)}dx =\sum_{j=1}^m \int_0^1 g_j(x)\overline{h_j(x)}dx
\end{equation*}
holds for all $h\in V_0$. This implies that $Af=-g$, and the proof is complete.
\end{proof}

By the above results, the operator $A$ generates on $X_2$ a semigroup $(T_2(t))_{t\geq 0}$: thus, the abstract Cauchy problem ${\rm(ACP)}$ (and hence the concrete diffusion problem on the network) is well-posed in $X_2$.
We can characterize several features of $(T_2(t))_{t\geq 0}$ by those of $(e^{tB})_{t\geq 0}$, and hence of the scalar matrix $B$: we are going to show that $(T_2(t))_{t\geq 0}$ is real, positive, and $X_\infty$-contractive  (i.e., the unit ball of $X_\infty$ is invariant under $T_2(t)$ for all $t\geq 0$), respectively, if and only if the semigroup $(e^{tB})_{t\geq 0}$ generated by $B$ is  real, positive, and $\ell^\infty$-contractive (i.e., once endowed ${\mathbb C}^{n-1}$ with the equivalent $\infty$-norm, $e^{tB}$ leaves the unit ball of ${\mathbb C}^{n-1}$ invariant for all $t\geq 0$), respectively.

\begin{theo}\label{bd}
The semigroup $(T_2(t))_{t\geq 0}$ on $X_2$ associated with $\ea$ enjoys the following properties:
\begin{itemize}
\item it is real if and only if the matrix $B$ has real entries;
\item it is positive if and only if the matrix $B$ has real entries that are positive off-diagonal,
\end{itemize}
If moreover $B$ is dissipative, then $(T_2(t))_{t\geq 0}$  is $X^\infty$-contractive if and only if
\begin{equation}\label{xinfty1}
{\rm Re} b_{ii} + \sum_{h\not=i} \vert b_{ih}\vert \leq 0,\qquad i=1,\ldots,n-1.
\end{equation}
\end{theo}

\begin{proof}
By Corollary~\ref{sa+diss} $\ea$ is densely defined, continuous, and elliptic. Thus, by~\cite[Prop. 2.5 and Thm. 2.6]{Ou04} and Proposition~\ref{main}, and taking into account the rescaling argument discussed in~\cite[\S~9.2]{Ar06}, the semigroup associated to $\ea$ is real and positive, respectively, if and only if
\begin{itemize}
\item ${f}\in V_0 \Rightarrow {\rm Re}f\in V_0 \hbox{ and } {\ea}({\rm Re}{f},{\rm Im}{f})\in\mathbb{R}$, and
\item ${f}\in V_0\Rightarrow ({\rm Re} {f})^+ \in V_0,\; a({\rm Re}f,{\rm Im}f)\in{\mathbb R}$, and $\ea(({\rm Re}f)^+,({\rm Re}f)^-)\leq 0$,
\end{itemize}
respectively.

Furthermore, if $B$ is dissipative, then by Proposition~\ref{main} $\ea$ is accretive and by~\cite[Thm.~2.15]{Ou04} the semigroup $(T_2(t))_{t\geq 0}$ is $X_\infty$-contractive if and only if
$$f\in V_0 \Rightarrow (1\wedge \vert {f}\vert){\rm sign} f \in V_0 \hbox{ and } {\rm Re}\ea((1\wedge \vert f\vert){\rm sign} f,(\vert f\vert-1)^+ {\rm sign} f)\geq 0,$$
where ${\rm sign}$ denotes the generalized (complex-valued)  function defined by
$${\rm sign} f(x):=\left\{
\begin{array}{ll}
\frac{f(x)}{\vert f(x)\vert},\quad &\hbox{if } f(x)\not=0,\\
0, &\hbox{if } f(x)=0.
\end{array}
\right.$$

If $k\in H^1(0,1)$, then it is clear that ${\rm Re}\; k\in H^1(0,1)$ as well as $({\rm Re}\; k)^+\in H^1(0,1)$. Furthermore, $({\rm Re}\; k)'={\rm Re}(k')$ and $(({\rm Re}\; k)^+)'={\rm Re}(k'){\mathbb 1}_{\{k\geq 0\}}$. Moreover the functions defined by
$$((1\wedge \vert k\vert){\rm sign} k)(x)=\left\{
\begin{array}{ll}
k(x),\quad &\hbox{if } \vert k(x)\vert \leq 1,\\
\frac{k(x)}{\vert k(x)\vert}, &\hbox{if } \vert k(x)\vert \geq 1,
\end{array}
\right.$$
as well as
$$((\vert k\vert-1)^+{\rm sign} k)(x)=\left\{
\begin{array}{ll}
0,\quad &\hbox{if } \vert k(x)\vert \leq 1,\\
k(x)-\frac{k(x)}{\vert k(x)\vert}, &\hbox{if } \vert k(x)\vert \geq 1,\end{array}
\right.$$
are in $H^1(0,1)$, with $((1\wedge \vert k\vert){\rm sign} k)'=k'{\mathbb 1}_{\{\vert k\vert\leq 1\}}$ and $((\vert k\vert-1)^{+}{\rm sign} k)'=k'{\mathbb 1}_{\{\vert k\vert\geq 1\}}$.

By definition, the subspace $V_0$ contains exactly those functions on the network that are continuous in the vertices and vanish in the vertex $\mv_{n}$. Then for every $f\in V_0$ we have ${\rm Re} ({f_j})=({\rm Re } f)_j$, $1\leq j\leq m$. It follows from the above arguments that ${\rm Re} f\in \left(H^1(0,1)\right)^m$, and one can see that the continuity of the values attained by $f$ in the vertices is preserved after taking the real part ${\rm Re} {f}$. All in all, ${\rm Re} f\in V_0$. Moreover, $\ea({\rm Re}{f},{\rm Im}{f})$ is the sum of $m$ integrals and $n$ numbers. Recall that the weights $c_1,\ldots c_m$ are real-valued, positive functions. Since $d^{{\rm Re} f}, d^{{\rm Im} f}\in{\mathbb R}^{n}$, it follows that ${\ea}({\rm Re}{f},{\rm Im}{f})\in\mathbb R$ if and only if
$$\sum_{i,h=1}^{n-1} b_{ih} d^{{\rm Re} f}_h{d^{{\rm Im} f}_i}\in\mathbb R.$$
This holds for all $f\in V_0$ if and only if $b_{ih}\in{\mathbb R}$.

Moreover, if $f\in V_0$, then $(({\rm Re} f)^+)_j = ({\rm Re} (f_j))^+$, $1\leq j\leq m$, and one sees as above that $({\rm Re} f)^+\in V_0$. In particular, for all $i=1,\ldots,n-1$ there holds
$$d_i^{({\rm Re}f)^+}=
\left\{
\begin{array}{ll}
0,\quad &\hbox{if } {\rm Re}d^f_i \leq 0,\\
{\rm Re}d^f_i,  &\hbox{if } {\rm Re}d^f_i  \geq 0,\end{array}
\right.
\quad\hbox{and}\quad
d_i^{({\rm Re}f)^-}=
\left\{
\begin{array}{ll}
-{\rm Re}d^f_i,  &\hbox{if } {\rm Re}d^f_i  \leq 0,\\
0,\quad &\hbox{if } {\rm Re}d^f_i \geq 0.\end{array}
\right.
$$

Accordingly,
 \begin{eqnarray*}
 \ea(({\rm Re} f)^+,({\rm Re} f)^-)&=&-\sum_{i,h=1}^{n-1} b_{ih} d^{({\rm Re}f)^+}_hd^{({\rm Re}f)^-}_i\\
 &=& -\sum_{\left\{\begin{array}{c}_{1\leq i,h\leq 2}\\_{i\not=h}\end{array}\right\}} b_{ih} ({\rm Re}d^f_h)^+({\rm Re}d^f_i)^-.
 \end{eqnarray*}
Thus, $ \ea(({\rm Re} f)^+,({\rm Re} f)^-)\leq 0$ if $B$ has positive off-diagonal entries. Conversely assume that $b_{i_0h_0}<0$ for some $i_0\not= h_0$. Then consider the function $f$ with boundary values $d^f_{i_0}=-1$, $d^f_{h_0}=1$, and $d^f_i=0$ for all other $i$. Then $ \ea(({\rm Re} f)^+,({\rm Re} f)^-)=-b_{i_0 h_0}>0$.

Finally, we discuss the property of $X_\infty$-contractivity for the semigroup associated to $\ea$. Thus, take $f\in V_0$. Then
$(1\wedge \vert f\vert){\rm sign}f\in \left(H^1(0,1)\right)^m$ and, again, the continuity of $f$ in the vertices imposes the same property to the function $(1\wedge \vert f\vert){\rm sign}f$, and in fact
$$d_i^{(1\wedge \vert f\vert){\rm sign}f}=(1\wedge \vert d^f_i\vert){\rm sign}d^f_i =\left\{
\begin{array}{ll}
d^f_i,\quad &\hbox{if } \vert d^f_i\vert \leq 1,\\
{\rm sign}d^f_i,  &\hbox{if } \vert d^f_i\vert  > 1,\end{array}
\right.$$
as well as
$$d_i^{(\vert f\vert-1)^+{\rm sign}f}=(\vert d^f_i\vert-1)^+{\rm sign}d^f_i=\left\{
\begin{array}{ll}
0,\quad &\hbox{if } \vert d^f_i\vert \leq 1,\\
d^f_i  - {\rm sign} d^f_i,  &\hbox{if } \vert d^f_i\vert  > 1,\end{array}
\right.$$
for all $i=1,\ldots,n-1$.
%
Now a direct computation yields
\begin{eqnarray*}
\ea((1\wedge \vert f\vert){\rm sign} f,(\vert f\vert-1)^+ {\rm sign} f)
&=&
 -\sum_{i,h=1}^{n-1} b_{ih} (1\wedge \vert d^f_h\vert){\rm sign}d^f_h \overline{(\vert d^f_i\vert-1)^+{\rm sign}d^f_i}\\
&=& b((1\wedge \vert d^f\vert){\rm sign}d^f, {(\vert d^f\vert-1)^+{\rm sign}d^f}),
\end{eqnarray*}
where $b$ denotes the sesquilinear, accretive form on ${\mathbb C}^{n-1}$ associated to the matrix $B$. Since the nodal values $d^f_1,\ldots,d^f_{n-1}$ are arbitrary,  by~\cite[Thm.~2.15]{Ou04} (which of course also applies to the accretive form $b$) one sees that the property of $X_\infty$-contractivity for $(T_2(t))_{t\geq0}$ is equivalent to that of $\ell^\infty$-contractivity for the semigroup $(e^{tB})_{t\geq 0}$ generated by $B$ on ${\mathbb C}^{n-1}$. Now the claim follows by Lemma~\ref{wolfg}.
\end{proof}

We recall that if $(T(t))_{t\geq 0}$ and $(S(t))_{t\geq 0}$ are semigroups on a Banach lattice $X$, then $(T(t))_{t\geq 0}$ is said to \emph{dominate $(S(t))_{t\geq 0}$ in the sense of positive semigroups} if
$$\vert S(t)f\vert \leq T(t)\vert f\vert \qquad\hbox{for all\;}f\in X,\; t\geq 0.$$

\begin{prop}\label{domin}
Let $B,\tilde{B}$ be $(n-1)\times (n-1)$ dissipative matrices and denote by $(e^{tB})_{t\geq 0}$ and $(e^{t\tilde{B}})_{t\geq 0}$ the semigroups they generate on ${\mathbb C}^{n-1}$. Assume that $B$ is real and positive off-diagonal, so that $(e^{tB})_{t\geq 0}$  is positive. Denote by $\ea_{B}$, $\ea_{\tilde {B}}$ the coercive form $\ea$ with coefficients given by $B$ and $\tilde{B}$, respectively, and by $(T_{B}(t))_{t\geq 0}$, $(T_{\tilde{B}}(t))_{t\geq 0}$ the associated semigroups on $X_2$. Then
$(T_{B}(t))_{t\geq 0}$ dominates $(T_{\tilde{B}}(t))_{t\geq 0}$ in the sense of positive semigroups if and only if $(e^{tB})_{t\geq 0}$ dominates $(e^{t\tilde{B}})_{t\geq 0}$ in the sense of positive semigroups.
\end{prop}

\begin{proof}
Observe that $V_0$, the domain of both $\ea_B$ and $\ea_{\tilde{B}}$ is an ideal of itself by~\cite[Prop.~2.20]{Ou04}. A direct computation shows that
$$\ea_{B}(\vert f\vert ,\vert g\vert) \leq {\rm Re}\ea_{{\tilde B}}(f,g)\qquad\hbox{for all } f,g\in V_0\hbox{ s.t. } f\overline{g}\geq 0$$
if and only if
$$\sum_{i,h=1}^{n-1} b_{ih} \vert d^f_h\vert \vert{d^g_i}\vert \geq {\rm Re}\sum_{i,h=1}^{n-1} \tilde{b}_{ih} d^f_h\overline{d^g_i}\qquad\hbox{for all } f,g\in V_0\hbox{ s.t. } f\overline{g}\geq 0.$$
Due to the arbitrarity of the nodal values of $f,g\in V_0$, such a condition is satisfied if and only if
$$\sum_{i,h=1}^{n-1} b_{ih} \vert x_h\vert \vert{y_i}\vert \geq {\rm Re}\sum_{i,h=1}^{n-1} \tilde{b}_{ih} x_h\overline{y_i}\qquad\hbox{for all } x,y\in {\mathbb C}^{n-1}\hbox{ s.t. } x\overline{y}\geq 0.$$
Thus, applying~\cite[Thm. 2.21]{Ou04} to the forms $\ea_B,\ea_{\tilde{B}}$ as well as to the forms associated to the matrices $B,\tilde{B}$ yields the claim.
\end{proof}

It has been shown in~\cite[\S~5]{KMS06} that the positive semigroup governing a diffusion problem on a network \emph{without Dirichlet conditions on any node} is irreducible if and only if $G$ is connected. This is no more true in the setting considered in this paper, unless we replace the notion of connectedness by a stronger one.




\begin{prop}
Let $B$ be dissipative with real, positive off-diagonal entries.
Then the following hold.
\begin{enumerate}
\item If the positive semigroup $(T_2(t))_{t\geq 0}$ is not irreducible, then $G$ is the union of two non-trivial subgraphs $G_1,G_2\subset G$ containing vertices $v_1,\ldots,v_{n_0}$ and $v_{n_0+1},\ldots, v_{n-1}$, respectively, and such that  $G_1\cap G_2\subset \{\mv_{n}\}$.
\item The converse also holds, if further $B$ is assumed to be block-diagonal, i.e., of the form
$$B=\begin{pmatrix}
b_{11} & \cdots & b_{1n_0} & \\
\vdots & \ddots & \vdots& &0 \\
b_{n_01} & \cdots & b_{n_0 n_0}\\
& & & b_{n_0+1\; n_0+1} & \cdots & b_{n_0+1\; n-1}\\
&0 && \vdots & \ddots & \vdots\\
&&&b_{n\;n_0+1} & \cdots & b_{n-1\; n-1}\end{pmatrix}$$
\item Let $B$ strictly positive off-diagonal. Then $(T(t))_{t\geq 0}$ is irreducible.
\end{enumerate}
\end{prop}

\begin{proof}
Recall that by~\cite[Thm. 2.10]{Ou04} the positive semigroup $(T_2(t))_{t\geq 0}$ is irreducible if and only whenever $\tilde{G}$ is an open subset of $G$ such that
\begin{equation}\label{irrcrit}
f\mathbb{1}_{\tilde{G}}\in V_0\qquad\hbox{and}
\qquad\hbox{Re}\ea(f\mathbb{1}_{\tilde{G}},f\mathbb{1}_{G\setminus \tilde{G}})\geq 0\qquad \hbox{for all } f\in V_0,
\end{equation}
then either $\tilde{G}$ or its complement has zero measure.

In particular, let $(T_2(t))_{t\geq 0}$ be irreducible and let
$f\in V_0$ be a function that agrees with the constant 1 on all
edges of $\tilde G$ that are \emph{not} incident to $\mv_{n}$, and
is smooth elsewhere. Then $f\mathbb{1}_{\tilde G}$ is of class
$H^1$ on each edge, continuous on each node of $G_1$ and vanishes
in $\mv_{n}$. Since $f\mathbb{1}_{\tilde G}$ also vanishes on all
nodes of $G_2$, it follows that $f\mathbb{1}_{\tilde G}\in V_0$.
We deduce that if $\tilde{G}$ contains an interior point $x$ of
some edge $e_j$, $j\not\in\Gamma(\mv_{n})$, then the whole $e_j$
belongs to $\tilde{G}$ (otherwise $f\mathbb{1}_{\tilde G}$ would
be discontinuous at $x$, and in particular not of class $H^1$.)

1) Let $(T(t))_{t\geq 0}$ be not irreducible. We want to show that $G$ is the union of nontrivial subgraphs whose intersection is $\{\mv_{n}\}$ or, in other words, that for all pair of points $x,y\in G$ every path $p(x,y)$ connecting them contains $\mv_{n}$.
Since the semigroup is not irreducible, there exists $\tilde{G}\subset G$ such that $\mu(\tilde{G})\not=0\not=\mu(G\setminus \tilde{G})$ and such that~\eqref{irrcrit} holds.
As remarked before, $\tilde{G}$ and $G\setminus \tilde{G}$ can thus only contain whole edges: let thus $e_{j_0},e_{\ell_0}$ be edges contained in $\tilde{G}$ and $G\setminus \tilde{G}$, respectively.
Assume now that then there exists a path $p\subset G$ that connects some interior point of $e_{j_0}$ to some interior point of $e_{\ell_0}$ and such that $\mv_{n}\notin p$. However, by a continuity argument the function $f\mathbb{1}_{\tilde G}$ would be constant 1 along the path $p$, a contradiction. Here we have denoted by $f$ a smooth function in $V_0$ that agrees with the constant 1 everywhere beside on the edged incident to $\mv_{n}$.

$2)$ Take $\tilde{G}=G_1$ and let $f\in V_0$. Then $f\mathbb{1}_{G_1}$ is a function that equals $f$ on the edges of $G_1$: by definition, $f\mathbb{1}_{G_1}$ is continuous in the vertices of $G_1$ and it vanishes in $\mv_{n}$. Furthermore, $f\mathbb{1}_{G_1}$ vanishes on the edges of $G_2$ and all the vertices adjacent to them, thus in particular it is continuous in the vertices of $G_2$, too. Summing up, $f\mathbb{1}_{G_1}\in V_0$ and moreover
$${\rm Re}\ea(f\mathbb{1}_{G_1},f\mathbb{1}_{G\setminus G_1})=-{\rm Re}\sum_{i=1}^{n_0} \sum_{h=n_0+1}^{n-1} b_{ih}d^f_h d^f_i = 0.$$
However, $G_1$ is not a set of measure zero.

$3)$ Let $\tilde{G}\subset G$ with $\mu(\tilde{G})\not=0$, and assume that~\eqref{irrcrit} holds. We are going to show that $\mu(G\setminus \tilde{G})=0$. In fact, let $\mu(G\setminus \tilde{G})\not=0$, i.e., let $G\setminus\tilde{G}$ contain (at least) one whole interval. We can thus assume that there exist nodes $v_{i_0}\in \tilde{G}$ and $v_{h_0}\in G\setminus \tilde{G}$. Let us now consider some function $f\in V_0$ such that $d^f_{i_0}=d^f_{h_0}=1$ and $d^f_{i}=0$ for all $i\not= i_0$, $i\not= h_0$. Then there holds
$$\ea(f\mathbb{1}_{\tilde{G}},f\mathbb{1}_{G\setminus \tilde{G}})= -b_{i_0 h_0} d^{f\mathbb{1}_{\tilde{G}}}_{h_0} d^{f\mathbb{1}_{G\setminus \tilde{G}}}_{i_0}=-b_{i_0 h_0}< 0,$$
a contradiction to~\eqref{irrcrit}.
\end{proof}



\section{Extrapolating semigroups, ultracontractivity, and Gaussian estimates}

Let the matrix $B$ and hence its adjoint $B^*$ be dissipative. One can easily see that the adjoint of the form $\ea$ is given by the form $\ea^*$ defined by
\begin{equation*}
\ea^*(f,g):=\sum_{j=1}^m\int_0^1 c_j(x) f'_j(x) \overline{g'_j(x)} dx -\sum_{i,h=1}^{n-1} \overline{b_{hi}} d^f_h\overline{d^g_i},\qquad f,g\in V_0.
\end{equation*}
By Theorem~\ref{bd} the semigroup associated to $\ea^*$ is $X_\infty$-contractive if and only if
\begin{equation}\label{xinfty2}
{\rm Re} b_{ii} + \sum_{h\not=i} \vert b_{hi}\vert \leq 0,\qquad i=1,\ldots,n-1.
\end{equation}
By duality, this is the case if and only if $(T_2(t))_{t\geq 0}$ is $X_1$-contractive. By interpolation we thus conclude that $(T_2(t))_{t\geq 0}$ is $X_p$-contractive for all $p\in[1,\infty]$ if and only if both~\eqref{xinfty1} and~\eqref{xinfty2} are satisfied. We thus obtain the following.

\begin{theo}\label{comp}
Let the matrix $B$ be dissipative, and let it satisfy the assumptions \eqref{xinfty1}--\eqref{xinfty2}. Then $(T_2(t))_{t\geq 0}$  extrapolates to a family of contractive semigroups $(T_p(t))_{t\geq 0}$ on $X_p$, $p\in[1,\infty]$, which are strongly continuous for $p\in[1,\infty)$ and analytic for $p\in(1,\infty)$.

 Furthermore, if $B$ has real (resp., real and positive off-diagonal) entries, then $(T_p(t))_{t\geq 0}$ is real for $p\in[1,\infty]$ (resp., positive for $p\in[1,\infty]$).

Finally, for $p\in (1,\infty)$ the spectrum of $A_{p}$ is independent of $p$, where  $A_{p}$ denotes the generator of $(T_p(t))_{t\geq 0}$, and $(T_p(t))_{t\geq 0}$ is uniformly exponentially stable with common growth bound $\omega(T_p)$ given by $s(A)$.
\end{theo}

\begin{proof}
We are in the position to apply the results summarized, e.g., in~\cite[\S~7.2]{Ar04} and deduce the existence of an extrapolation semigroup $(T_p(t))_{t\geq 0}$ on $X_p$, $1\leq p\leq \infty$. A list of the all properties of $(T_2(t))_{t\geq 0}$ inherited by $(T_p(t))_{t\geq 0}$ can be found in~\cite[\S~7.2.2]{Ar04}. This yields all the claimed properties up to exponential stability. To check this, we combine the uniform exponential stability of the semigroup associated with $\ea$, cf. Corollary~\ref{sa+diss}, and the $p$-independence of the spectrum of the analytic semigroups $(T_p(t))_{t\geq 0}$, $p\in (1,\infty)$. Hence, the growth bound of $(T_p(t))_{t\geq 0}$ agrees with the spectral bound $s(A_p)=s(A_2)=s(A) <0$.
\end{proof}

\begin{rems}
1) Let $(T_2(t))_{t\geq 0}$ be positive and irreducible. If the semigroup extrapolates to $X_p$, $p\in [1,\infty]$, then it is known that also $(T_p(t))_{t\geq 0}$ are irreducible, $p\in[1,\infty)$, cf.~\cite[\S~7.2]{Ar04}.

2) 
Consider the case of $c_1=\ldots=c_m\equiv 1$ and $b_{ih}=0$, $i,h=1,\ldots,n$. Then, it follows by~\cite[Th\'eo. 2.4 and Th\'eo. 3.1]{Ni87} that
$$-\left(\frac{\pi}{m+1}\right)^1\leq s(A)=\omega(T_p)\leq -\left(\frac{\pi}{2m}\right)^2,\qquad p\in(1,\infty).$$
In other words, we can say that \emph{the more edges belong to the network, the slower is the heat dissipation}.
A further, more involved upper estimate on $s(A)$ is shown in~\cite[Th\'eo. 3.2]{Ni87}, showing that the inner structure of the graph does influence the asymptotical behavior of the diffusion problem.
\end{rems}

We have shown that, if $B$ is dissipative and satisfies \eqref{xinfty1}--\eqref{xinfty2}, then the semigroup $(T_p(t))_{t\geq 0}$ is contractive on $X_p$, $p\in [1,\infty]$. In fact, we can characterize a stronger property.

\begin{lemma}\label{ultralemma}
Let the matrix $B$ be dissipative.
\begin{enumerate}
\item If $B$ satisfies~\eqref{xinfty1}, then the semigroup $(T_2(t))_{t\geq 0}$ on $X_2$ associated with $\ea$ is ultracontractive. In particular, it satisfies the estimate
\begin{equation}\label{ultra}
\Vert T_2(t)f\Vert_{X_\infty} \leq  M t^{-\frac{1}{4}}\Vert f\Vert_{X_2}
\qquad\hbox{ for all }t>0,\; f\in{X_1},
\end{equation}
for some constant $M$.
\item If instead $B$ satisfies~\eqref{xinfty2}, then
\begin{equation}\label{ultrabis}
\Vert T_2(t)f\Vert_{X_2} \leq  M t^{-\frac{1}{4}}\Vert f\Vert_{X_1}
\qquad\hbox{ for all }t>0,\; f\in{X_1},
\end{equation}
holds.
\end{enumerate}
\end{lemma}

\begin{proof}
By~\cite[Thm. 6.3]{Ou04} it suffices to show that there holds
\begin{equation*}\label{ultra1b}
\Vert f\Vert_{X_2}\leq M \ea(f,f)^{\frac{1}{6}} \cdot \Vert f\Vert^{\frac{2}{3}}_{X_1}\qquad\hbox{for all } f\in V_0,
\end{equation*}
for some constant $M$. Recall that
\begin{eqnarray*}
\Vert k\Vert_{L^2(0,1)} &\leq& M_1 \left(\Vert k'\Vert_{L^2(0,1)}+\Vert k\Vert_{L^1(0,1)}\right)^{\frac{1}{3}} \cdot \Vert k\Vert_{L^1(0,1)}^{\frac{2}{3}}\\
&\leq & M_1 \Vert k\Vert_{H^1(0,1)}^{\frac{1}{3}} \cdot \Vert k\Vert_{L^1(0,1)}^{\frac{2}{3}},
\end{eqnarray*}
is valid for all $k\in H^1(0,1)$ and some constant $M_1$, cf. \cite[Thm. 1.4.8.1]{Ma85}.

Take finally $f\in V_0$ and observe that by the above Nash-type inequality
\begin{eqnarray*}
\Vert f\Vert_{X_2}^2&=& \sum_{j=1}^m \Vert f_j\Vert_{L^2(0,1)}^2 \leq
M^2_1 \sum_{j=1}^m \Vert f_j\Vert_{H^1(0,1)}^{\frac{2}{3}} \cdot \Vert f_j\Vert_{L_1(0,1)}^{\frac{4}{3}}\\
&\leq & M_2 \left(\sum_{j=1}^m \Vert f_j\Vert_{H^1(0,1)}\right)^{\frac{2}{3}} \cdot \left(\sum_{j=1}^m\Vert f\Vert_{L^1(0,1)}\right)^{\frac{4}{3}}\\
&\leq & M_3 \Vert f\Vert^{\frac{2}{3}}_{V_0} \cdot \Vert f\Vert_{X_1}^{\frac{4}{3}}\\
&\leq & M_3 \ea(f,f)^{\frac{1   }{3}} \cdot \Vert f\Vert_{X_1}^{\frac{4}{3}}.
\end{eqnarray*}
We have so far shown that
$$\Vert T_2(t)f\Vert_{X_\infty} \leq  M t^{-\frac{1}{4}}\Vert f\Vert_{X_2}
\qquad\hbox{ for all }t>0,\; f\in{X_2}.$$
If instead~\eqref{xinfty2} holds, then the claim follows by duality.
\end{proof}

\begin{rem}
In the terminology of N.Th. Varopoulos (\cite[\S~0.1]{Va85}, cf. also~\cite[\S~7.3.2]{Ar04}), Lemma~\ref{ultralemma} says that the dimension of the semigroup $(T_2(t))_{t\geq 0}$ is 1. This is true regardless of the structure of the underlying graph.
\end{rem}

Combining the stability and ultracontractivity results we have obtained, we can finally show that $X_2\!-\!X_\infty$ uniform exponential stability holds.

\begin{cor}\label{ultralemma3}
Let the matrix $B$ be dissipative and satisfy~\eqref{xinfty1}.   Then the semigroup $(T_2(t))_{t\geq 0}$ on $X_2$ satisfies the estimate
\begin{equation*}
\Vert T_2(t) f\Vert_{X_\infty} \leq
M \left(\frac{1-t\omega}{t}\right)^{\frac{1}{4}}e^{ t\omega}\Vert f\Vert_{X_2}
\qquad\hbox{ for all }t>0,\; f\in{X_2},
\end{equation*}
where $M$ is the constant that appears in~\eqref{ultra} and $\omega$ is the strictly negative growth bound of $(T_2(t))_{t\geq 0}$.
\end{cor}

\begin{proof}
Taking into account Lemma~\ref{ultralemma} and Theorem~\ref{comp}, the claim follows directly from~\cite[Lemma~6.5]{Ou04}.
\end{proof}

In order to discuss the well-posedness of the problem in an $L^p$-setting, we want to identify the generators of the  $(T_p(t))_{t\geq 0}$.

\begin{prop}\label{identp}
Let the matrix $B$ be dissipative, and let it satisfy the assumptions \eqref{xinfty1}--\eqref{xinfty2}. Then for all $p\in[1,\infty]$ the generator $A_{p}$ of the semigroup $(T_p(t))_{t\geq 0}$ is the operator whose action on the domain
\begin{equation}\label{dominap}
D(A_p):=\left\{f\in \left(W^{2,p}(0,1)\right)^m:
\hspace{-10pt}
\begin{array}{rl}
&  \exists d^f\in {\mathbb C}^{n-1}\times \{0\} \hbox{ s.t. } \\
&(\Phi^+)^\top d^f=f(0),\; (\Phi^-)^\top d^f=f(1),\\
&\hbox{ and }\Phi_w^+ f'(0)-\Phi_w^- f'(1)={\mathbb B}d^f\\
\end{array}
\right\}.
\end{equation}
is formally given in~\eqref{amain}. In particular, $A_p$ has compact resolvent for all $p\in[1,\infty]$.
\end{prop}


\begin{proof}
The proof goes in two steps. We first consider the case of $p\in (2,\infty]$, then discuss the case $p\in [1,2)$ by duality.

1) Let $p>2$. We have already remarked that $X_p\hookrightarrow X_q$ for all $1\leq q\leq p\leq\infty$. Moreover, for all $p>2$   the space $X_p$ is invariant under $(T_p(t))_{t\geq 0}$ by the ultracontractivity of $(T_2(t))_{t\geq 0}$, $p>2$. Thus, by~\cite[Prop. II.2.3]{EN00} the generator of $(T_p(t))_{t\geq 0}$ is the part of $A$ in $X_p$. A direct computations yields the claim.

2) Take now some $p$ with $1\leq p<2$, and let $q$ such that $p^{-1}+q^{-1}=1$. By an argument similar to that of~\cite[Thm. 1.4.1]{Da89} one has that the adjoint semigroup $((T_p)^*(t))_{t\geq 0}$ of $(T_p(t))_{t\geq 0}$ on $X_p$ is actually the extrapolation semigroup $((T^*)_q(t))_{t\geq 0}$ of $((T^*)_2(t))_{t\geq 0}$. Since the generator $(A^*)_q$ of $((T^*)_q(t))_{t\geq 0}$ also satisfies the assumption of the theorem, by 1) we deduce that
\begin{equation*}
D((A^*)_q):=\left\{f\in \left(W^{2,q}(0,1)\right)^m:
\hspace{-10pt}
\begin{array}{rl}
&  \exists d^f\in {\mathbb C}^{n-1}\times \{0\} \hbox{ s.t. } \\
&(\Phi^+)^\top d^f=f(0),\; (\Phi^-)^\top d^f=f(1),\\
&\hbox{ and }\Phi_w^+ f'(0)-\Phi_w^- f'(1)={\mathbb B}^*d^f\\
\end{array}
\right\}.
\end{equation*}
 Set
\begin{equation*}
D_p:=\left\{f\in \left(W^{2,p}(0,1)\right)^m:
\hspace{-10pt}
\begin{array}{rl}
&\exists d^f\in {\mathbb C}^{n-1}\times \{0\} \hbox{ {s.t.}} \\
&(\Phi^+)^\top d^f=f(0),\; (\Phi^-)^\top d^f=f(1),\\
&\hbox{ and }\Phi_w^+ f'(0)-\Phi_w^- f'(1)={\mathbb B}d^f\\
\end{array}
\right\}.
\end{equation*}
Consider the operator $A_p$ whose action on $D_p$ is formally given by~\eqref{amain}. We are going to show that the adjoint $(A_p)^*$ of $(A_p,D_p)$ actually agrees with $(A^*)_q$. Then, since the generator of the pre-adjoint semigroup $(T_p(t))_{t\geq 0}$ on $X_p$ of $((T^*)_q(t))_{t\geq 0}$ on $X_q$ is the pre-adjoint operator $A_p$ of $A_q$, we conclude that $(A_p,D_p)$ generates the $C_0$-semigroup $(T_p(t))_{t\geq 0}$ on $X_p$, and the claim follows.

By definition we have that the adjoint of $(A_p,D_p)$ is the operator $((A_p)^*,D^*_p)$ given by
\begin{eqnarray*}
D^*_p&=&\{ f\in X^*_p: \exists g\in X_p^*: <\!A_p u,f\!>\;=\;<\!u,g\!> \; \forall u\in D_p\}\\
&=&\{ f\in X_q: \exists g\in X_q: <\!A_p u,f\!>\;=\;<\!u,g\!> \; \forall u\in D_p\},\\
(A_p)^* f&=& g.
\end{eqnarray*}
Let us first show that $D^*_p \subset D((A^*)_q)$. Take $f\in D^*_p$ and observe that the identity $<\!A_p u,f\!>\;=\;<\!u,g\!>$ holds in particular for all $u$ of the form introduced in~\eqref{formfun}, with $u_j\in C^\infty_c(0,1)$. Thus, we obtain that for all $j=1,\ldots, m$ there exists $g_j\in L^q(0,1)$ such that
$$\int_0^1 (c_j u_j')'(x)\overline{f_j(x)}dx=\int_0^1 u_j(x)\overline{g_j(x)}dx.$$
Integrating by parts one thus obtains that $(c_j f_j')'=g_j$ (in
the sense of distributions), and since $g_j\in L^q(0,1)$ it
follows from the definition of Sobolev space of order 2 that
$f_j\in W^{2,q}(0,1)$. Thus, we conclude that $f\in
\left(W^{2,q}(0,1)\right)^m$.

In order to check that the node conditions are also verified, let
us perform a computation similar to that in Lemma~\ref{ident}. The
condition $<\!A_p u,f\!>\;=\;<\! u,g\!>$ for $u\in D_p$ reads then
\begin{equation}\label{azero}
\begin{array}{rl}
&\sum_{j=1}^m \int_0^1 u_j(x) \overline{\left((c_j f_j')'(x)-g_j(x)\right)} dx =\\
&\qquad\qquad\quad -\sum_{j=1}^m c_j(\mv_{n})\phi_{nj} u'_j(\mv_{n}) \overline{f_j(\mv_n)}\\
&\qquad\qquad\qquad +\sum_{i=1}^{n-1} \sum_{j=1}^m (\omega_{ij}^- -\omega_{ij}^+) \left( d^u_i \overline{f'_j(\mv_i)}-u'_j(\mv_i) \overline{f_j(\mv_i)}\right)
\end{array}
\end{equation}
for some $g\in X_q$. Since $u\in D_p$ is arbitrary (and therefore so are its derivative's nodal values), we deduce that all terms on both the left and the right hand sides must vanish identically. In particular, whenever the edge $e_j$ is incident to $\mv_{n}$ the number $c_j(\mv_{n})\phi_{nj}u'_j(\mv_{n})$ is arbitrary (recall that $c_j(x)\geq c>0$ for all $x\in[0,1]$ and $j=1,\ldots,m$), so that necessarily $f_j(\mv_{n})=0$ for all $j\in \Gamma(\mv_{n})$. Now we invoke again the arbitrarity of $u\in D_p$ (and hence of its and its derivative's nodal values) and, considering functions $u$ s.t. $d^u_i=0$ for $i=1,\ldots,n$, we conclude that there holds 
\begin{equation}\label{contin}
\sum_{j=1}^m (\omega_{ij}^- -\omega_{ij}^+) u'_j(\mv_i) \overline{f_j(\mv_i)}=0\qquad\hbox{ for all }i=1,\ldots,n-1.
\end{equation}
Moreover, observe that the generalized Kirchhoff law for $u\in D_p\cap (H^1_0(0,1))^m$ becomes
\begin{equation}\label{cont2}
\sum_{j=1}^m (\omega_{ij}^- -\omega_{ij}^+) u'_j(\mv_i)=0\qquad\hbox{ for all }i=1,\ldots,n-1.
\end{equation}

Let us reduce~\eqref{contin} and~\eqref{cont2} to pairwise relations. More precisely, pick $u$ in such a way that only exactly two values $(\omega_{i\ell}^- -\omega_{i\ell}^+)u'_\ell(\mv_i),(\omega_{ik}^- -\omega_{ik}^+)u'_k(\mv_i)$ are non-zero, $1\leq \ell,k\leq m$. Then, we obtain from~\eqref{contin}--\eqref{cont2} that for given $i$ the vector
$$\begin{pmatrix}
(\omega_{i\ell}^- -\omega_{i\ell}^+)u'_\ell(\mv_i)\\
(\omega_{ik}^- -\omega_{ik}^+)u'_k(\mv_i)
\end{pmatrix}
\hbox{ is orthogonal to }
\begin{pmatrix}
\overline{f_\ell(\mv_i)}\\
\overline{f_k(\mv_i)}
\end{pmatrix}
\hbox{ as well as to }
\begin{pmatrix}
1\\
1
\end{pmatrix}.$$
This promptly yields that $f_\ell(\mv_i)=f_k(\mv_i)$. Repeating the argument $m-1$ times we conclude that the nodal values $f_j(\mv_i):=d^f_i$ do not depend on $j\in \Gamma(\mv_i)$.

Thus, the second term on the right hand side of~\eqref{azero} can be written as
$$\sum_{i=1}^{n-1} \left( d^u_i \sum_{j=1}^m (\omega_{ij}^- -\omega_{ij}^+) \overline{f'_j(\mv_i)}-
\overline{d^f_i} \sum_{j=1}^m (\omega_{ij}^- -\omega_{ij}^+) u'_j(\mv_i) \right)$$
or rather, taking into account the generalized Kirchhoff condition satisfied by $u$, as
\begin{equation}\label{done}
 \sum_{i=1}^{n-1} \left(d^u_i \sum_{j=1}^m (\omega_{ij}^-
 -\omega_{ij}^+) \overline{f'_  j(\mv_i)}-\overline{d^f_i} \sum_{h=1}^{n-1} b_{ih} {d^u_h} \right).
\end{equation}
Since this expression vanishes identically and because of the arbitrarity of the nodal values of $u$ and $u'$, we conclude that
$$\sum_{j=1}^m (\omega_{ij}^- -\omega_{ij}^+) {f'_j(\mv_i)}=\sum_{h=1}^{n-1} \overline{b_{hi}} {d^f_h},\qquad i=1,\ldots, n-1.$$
Summing up, $f$ satisfies the Dirichlet condition at $\mv_{n}$ as well as the generalized Kirchhoff law at $\mv_1,\ldots,\mv_{n-1}$ for coefficients given by $B^*$, thus $f\in D((A^*)_q)$.

Let us now check that $D((A^*)_q)\subset D^*_p$. Take $f\in \left(W^{2,q}(0,1)\right)^m$ satisfying the continuity condition on the nodal values as well as the Dirichlet condition on $\mv_{n}$ and the generalized Kirchhoff law on $\mv_1,\ldots,\mv_{n-1}$ for coefficients given by $B^*$. Set $g_j= (c_j f_j')'$, so that $g\in X_q$. We only have to prove that for all $u\in D_p$ there holds $<\!A_p u,f\!>\;=\;<\!u,g\!>$, i.e.,
$$\sum_{j=1}^m \int_0^1 (c_j u_j')'(x) \overline{f_j(x)} dx= \sum_{j=1}^m \int_0^1 u_j(x) \overline{(c_j f_j')'(x)} dx.$$
Integrating by parts as in the proof of the converse inclusion and recalling that the nodal values of both $f$ and $u$ do not depend on $j$ we see that this is the case if and only if the expression in~\eqref{done} vanishes identically.
\end{proof}


As a direct consequence of the ultracontractivity of $(T_2(t))_{t\geq 0}$ and the Dunford--Petty Theorem, the semigroup has an integral kernel for all $t>0$, cf. \cite[\S~7.3]{Ar04}. More precisely, denote by $(\tilde{T}_p(t))_{t\geq 0}$ the semigroup on $L^p(0,m)$ that is similar to $(T_p(t))_{t\geq 0}$ on $X_p$ with a similarity transformation given by the isometry $U$ introduced in Definition~\ref{xplp}. Then for all $p\in[1,\infty]$ the action of $(\tilde{T}_p(t))_{t\geq 0}$ is given by
$$\tilde{T}_p(t)g(\cdot)=\int_0^m \tilde{K}_t(\cdot,y)g(y)dy,\qquad t> 0,\; g\in L^p(0,m),$$
for a suitable kernel $\tilde{K}_t\in L^\infty\big((0,m)\times (0,m)\big)$.

The existence of integral kernels is a typical feature of diffusion equations. Also in view of its consequences  in the theory of evolution equations (see e.g.~\cite{Ar97}), it is of great interest to estimate such kernels and compare them with the standard Gaussian one, which is associated to the heat equation on the whole space. This is usually done by the so-called Davies' trick, that amounts to prove uniform $L^\infty$-(quasi)contractivity estimates for a class of perturbed semigroups.

More precisely, introduce a class $W$ of functions
$\varphi: {\mathbb R}\to {\mathbb R}$ constructed in the following way: first, we consider smooth functions $\phi\in \left(C^\infty_b (0,1) \right)^m$ such that
$$\exists d^\phi\in {\mathbb C}^{n} \hbox{ s.t. }(\Phi^+)^\top d^\phi=\phi(0),\; (\Phi^-)^\top d^\phi=\phi(1),\;\Vert \phi'\Vert_\infty\leq 1 \hbox{ and }\Vert \phi''\Vert_\infty\leq 1.$$
We then stretch via the isometry $U$  such functions over the network to functions $\tilde{\phi}:(0,m)\to {\mathbb R}$. We  finally call $W$ the class of all smooth, bounded continuous extensions $\varphi$ of $\tilde \phi$ to the whole line such that also $\Vert \varphi'\Vert_\infty\leq 1 \hbox{ and }\Vert \varphi''\Vert_\infty\leq 1$. Finally, for fixed $\varphi\in W$ we define an operator $L_\rho$ on $L^2(\mathbb R)$ by $L_\rho f:= e^{-\rho \varphi}f$ and perturbed semigroups $(\tilde{T}^\rho_2(t))_{t\geq 0}$, where $\tilde{T}^\rho_2(t):=L_\rho \tilde{T}_2(t) L^{-1}_\rho$, $\rho\in\mathbb R$. Observe that, by construction, $U^{-1}L_\rho Uf\in V_0$ for all $f\in V_0$, $\rho\in\mathbb R$.

In the remainder of this section we consider $X_p$ as real spaces, $1\leq p\leq \infty$. Then, by~\cite[Thm.~3.3]{AT97} Gaussian estimates for $(\tilde{T}_2(t))_{t\geq 0}$ are equivalent to ultracontractivity estimates for $(\tilde{T}^\rho_2(t))_{t\geq 0}$, uniformly in $\rho\in\mathbb R$ and $\varphi\in W$. This can be done by applying the above presented theory to the similar semigroups $({T}^\rho_2(t))_{t\geq 0}$ on $X_2$, which by a direct computation are associated to the bilinear forms $\ea^\rho$ defined by
\begin{eqnarray*}
\ea^\rho(f,g)&:=&\sum_{j=1}^m\int_0^1 c_j(x)f'_j(x) g'_j(x)dx\\
&&\quad + \rho \sum_{j=1}^m \int_0^1 c_j(x)\phi'_j(x) \left(f_j(x)g'_j(x) - f'_j(x) g_j(x)\right)dx\\
&&\quad - \rho^2 \sum_{j=1}^m \int_0^1 c_j(x)\phi'_j(x)^2 f_j(x)g_j(x)  dx\\
&&\quad-\sum_{i,h=1}^{n-1} b_{ih} e^{\rho (d^\phi_h- d^\phi_i)} d^f_h {d^g_i},
\end{eqnarray*}
for all $f,g\in V_0$. In the following we restrict ourselves to the local case, i.e., to the case where $B$ is a diagonal matrix. In fact, by considering the above form $\ea^\rho$ with coefficients $c\equiv 1$ and $B=\begin{pmatrix}-1 & 1\\ 1 & -1\end{pmatrix}$ on the domain $H^1(0,1)$ (i.e., for the sake of simplicity, on a graph with a single edge without Dirichlet boundary conditions), one sees that it is not possible to find $\omega\in\mathbb R$ such that the shifted form $\ea^\rho+\omega(1+\rho^2)$ is accretive uniformly in $\rho\in\mathbb R$ and $\phi\in W$, i.e., such that
\begin{eqnarray*}
0\leq \ea^\rho(f,f)+\omega(1+\rho^2)\Vert f\Vert^2_{X_2}&=&\int_0^1 f'(x)^2+(\rho^2\omega + \omega -\rho^2\phi'(x)^2) f(x)^2 dx \\
&&\quad - (e^{\rho (\phi(1)-\phi(0))}+e^{\rho(\phi(0)-\phi(1))})f(1)f(0)\\
&&\quad + f(0)^2 + f(1)^2
\end{eqnarray*}
for all $\rho\in\mathbb R$, $\phi\in W$, and $f\in H^1(0,1)$. This can be checked by taking $f$ constant, $\rho=1$, $\phi(x):=(1+\omega) x$.
The uniform accretivity of the forms $a^\rho$ seems to be an essential ingredient of the method of proof explained in~\cite{AT97}: we thus derive in the following Gaussian estimates only in the local case of $B$ diagonal.

\begin{theo}\label{gaussian}
Let the matrix $B$ be diagonal with negative entries. Then the semigroup $(T_2(t))_{t\geq 0}$ has Gaussian estimates. More precisely, there exist constants $b,c>0$ such that $\tilde{K}_t$ satisfies
\begin{equation}\label{gauss}
0\leq \tilde{K}_t(x,y)\leq c t^{-\frac{1}{2}} e^{-\frac{b\vert x-y\vert^2}{t}+t},\qquad x\in(0,m),
\end{equation}
uniformly in $t>0$.
\end{theo}

\begin{proof}
Under our assumptions it is already known that $(T_2(t))_{t\geq 0}$ is positive and contractive with respect to the $X_1$ and $X_\infty$ norm. Moreover, a direct computation shows that the shifted form $\ea^\rho + (1+\rho^2)$ is accretive.

Observe now that $V_0$ is not an ideal of $(H^1(0,1))^m$, but indeed of
$$V:=\left\{f\in \left(H^1(0,1)\right)^m:\!\!\!\!\!\!\!
\begin{array}{ll}
&\exists d^f\in {\mathbb C}^{n-1}\times \{0\} \hbox{ s.t. }\\
&(\Phi^+)^\top d^f=f(0)\hbox{ and } (\Phi^-)^\top d^f=f(1)
\end{array}\right\}.$$
The proof can now be concluded by mimicking~\cite[Thm.~4.4]{AT97}.
\end{proof}


We are finally able to obtain an optimal result on the analiticity of the semigroup generated by $A$. We stress that we are~\emph{not} imposing any assumption on $B$.

\begin{theo}\label{anal}
Consider the operator $A_p$ whose action on the domain defined in~\eqref{dominap} is formally given by~\ref{amain}. Then $A_p$ generates on $X_p$ a strongly continuous, analytic semigroup $(T_p(t))_{t\geq 0}$ of angle $\frac{\pi}{2}$, $p\in [1,\infty)$.
\end{theo}

\begin{proof}
The proof goes in three steps.

1) Let us first consider the case of $B=0$, $p=2$. Then, it follows by Proposition~\ref{main} that the form $a$ is symmetric and coercive, hence in particular the associated operator $A$ is self-adjoint and dissipative, and the statement follows by the spectral theorem.

2) If $B=0$, then for general $p\in [1,\infty]$ the semigroup $(T_2(t))_{t\geq 0}$ extrapolates to a semigroup $(T_p(t))_{t\geq 0}$ on $X_p$ that is analytic  of angle $\frac{\pi}{2}$, by Theorem~\ref{gaussian} and~\cite[Thm. 6.16]{Ou04}.

3) Finally, consider the case of general $B$. Then we can apply the theory developed in~\cite[\S~2]{GK91}, after setting $X:=X^p$, $\partial X:={\mathbb C}^{n-1}$, and
$$Y:=\left\{f\in \left(W^{2,p}(0,1)\right)^m:
\hspace{-10pt}
\begin{array}{rl}
& \exists d^f\in {\mathbb C}^{n-1}\times \{0\} \hbox{ \rm{s.t.}} \\
&(\Phi^+)^\top d^f=f(0),\; (\Phi^-)^\top d^f=f(1)\\
\end{array}
\right\},$$
as well as
$$Lu:=\begin{pmatrix}
\sum_{j=1}^m \phi_{1j} c_j(\mv_1) u'_j(\mv _1)\\
\vdots\\
\sum_{j=1}^m \phi_{{n-1}j} c_j(\mv_{n-1}) u'_j(\mv _{n-1})
      \end{pmatrix},\quad u\in D(A),$$
and
$$\Phi u:={\mathbb B}d^u,\quad u\in V_0.$$

We consider the operator $A$ with maximal domain $Y$ and observe that  the restriction of $A$ to ${\rm ker}(L)$ is the operator considered in 2), hence the generator of an analytic semigroup of angle $\frac{\pi}{2}$. Since the boundary perturbation operator $\Phi:V_0\to \partial X$ is compact, the claim follows by~\cite[Thm.~2.6]{GK91} (observe that in that theorem is also proved, although not explicitly stated, that the angle of analiticity remains invariant under admissible boundary perturbations).
\end{proof}

\begin{rem}
Gaussian estimates like~\eqref{gauss} are a key argument for
discussing a number of different issues that go far beyond the
scope of this paper. Without going into details, we recall that
Theorem~\ref{gaussian} implies at once, among other, the property
of maximal regularity for $(T_p(t))_{t\geq 0}$ for
$p\in(1,\infty)$, upper estimates for the time derivative of the
kernel $K_t$, $L^p$-estimates for Schr\"odinger and wave
equations, and the fact that $A_p$ has bounded $H^\infty$-calculus
on each sector (and therefore that it has bounded imaginary
powers) for $p\in(1,\infty)$,  cf. \cite[\S~6.5 and Chapt.
7]{Ou04},~\cite[\S~7.4]{Ar04} and references therein. Observe
that, even if Gaussian estimates can only be obtained for local
nodal conditions, most of the above mentioned consequences also
hold in the general case by perturbation methods.
\end{rem}

\begin{theo}\label{wellp}
The first order network diffusion problem introduced in Section~2 is well-posed on $X_p$, $p\in[1,\infty)$, i.e., for all initial data $u_0\in X_p$ it admits a unique mild solution that continuously depends on the initial data. Such a solution is of class $C^\infty$ in both variables $x,t$ and its $\infty$-norm tends to 0 in time.
If furthermore $c_j\in C^\infty[0,1]$, $j=1,\ldots,m$, then the solution $u(t,\cdot)$ is of class $C^\infty$ with respect to the space variable.
\end{theo}

\begin{proof}
The well-posedness and boundedness results follow from the fact that the semigroup $(T_2(t))_{t\geq 0}$ is ultracontractive and extends to a family of semigroups $(T_2(p))_{t\geq 0}$ that, by Proposition~\ref{identp}, actually govern the network diffusion problem. The decay of the solution is ensured by the uniform exponential stability of all semigroups.

Finally, if $c_j\in C^\infty[0,1]$, $j=1,\ldots,m$, then one sees that $D(A_{p}^\infty)\subset
\left(C^\infty[0,1]\right)^m$ for all $p\in (1,\infty)$. Since the semigroup $(T_p(t))_{t\geq 0}$ is analytic, it maps $X_p$ into $D(A_{p}^\infty)$, and the claim follows.
\end{proof}

Observe that if we replace the Dirichlet condition in $\mv_{n}$ by continuity of the values of $u_j(t,\mv_{n})$, $t\geq 0$, $j\in \Gamma(\mv_{n})$, plus a Kirchhoff-type condition analogous to that imposed on the other nodes, we obtain the system
\begin{equation}\label{netcpb}
\left\{
\begin{array}{rcll}
\dot{u}_j(t,x)&=& (c_j u_j')'(t,x), & x\in(0,1),\; j=1,\dots,m,\\
u_j(t,\mv _i)&=&u_\ell (t,\mv _i)=:d^u_i, & j,\ell\in \Gamma(\mv _i),\; i=1,\ldots,n,\\
\sum_{h=1}^{n} b_{ih} d^u_h&=& \sum_{j=1}^m \phi_{ij} c_j(\mv_i) u'_j(t,\mv_i), & i=1,\ldots,n,\\
\end{array}
\right.
\end{equation}
for $t\geq 0$, with initial conditions
$$u_j(0,x)=u_{0j}(x),\qquad x\in (0,1),\; j=1,\dots,m.$$
Here $b_{1h},b_{i1}$, are arbitrary numbers, $1\leq i,h\leq n-1$. Such an initial-value problem has been proved to be well-posed in~\cite{KMS06} (in the special case of $B=0$): we can compare its solution and that to the original network diffusion problem and obtain the following. For the sake of simplicity, in the following we restrict to the case of purely Kirchhoff nodal conditions. However, one can see that  a similar proof also works in the general case.

\begin{prop}
Let the coefficients $b_{ih}=0$, $1\leq i,h\leq n$. Then the
semigroup $(T_2(t))_{t\geq 0}$ governing the original network diffusion problem (as in~\S~2) is dominated
by the semigroup $(\tilde{T}_2(t))_{t\geq 0}$
governing~\eqref{netcpb} in the sense of positive semigroups.
\end{prop}

\begin{proof}
As shown in~\cite{KMS06}, $(\tilde{T}_2(t))_{t\geq 0}$ is a sub-Markovian semigroup that comes from a form with domain
$$V=\left\{f\in \left(H^1(0,1)\right)^m:\exists d^f\in {\mathbb C}^{n} \hbox{ s.t. } (\tilde{\Phi}^+)^\top d^f=f(0)\hbox{ and } (\tilde{\Phi}^-)^\top d^f=f(1)\right\},$$
where $\tilde{\Phi}^+=(\tilde{\phi}^+_{ij})$ and
$\tilde{\Phi}^-=(\tilde{\phi}^-_{ij})$ represent the incidence
matrices defined by
\begin{equation*}
\tilde{\phi}^+_{ij}:=
\left\{
\begin{array}{rl}
1, & \hbox{if } \me _j(0)=\mv _i,\\
0, & \hbox{otherwise},
\end{array}
\right.
\quad\hbox{and}\quad
\tilde{\phi}^-_{ij}:=
\left\{
\begin{array}{rl}
1, & \hbox{if } \me _j(1)=\mv _i,\\
0, & \hbox{otherwise}.
\end{array}
\right.
\end{equation*}

Since a Dirichlet condition in the node $\mv_{n}$ implies in particular continuity on a function in that vertex, one sees that $V_0\subset V$. Accordingly, by~\cite[Cor. 2.22]{Ou04} it suffices to prove that $V_0$ is an ideal of $V$, i.e., that the following conditions are satisfied:
\begin{itemize}
\item $f\in V_0\Rightarrow \vert f\vert \in V$,
\item $f\in V_0,\; g\in V, \hbox{ and } \vert g\vert\leq \vert f\vert\Rightarrow g\cdot {\rm sign}f\in V_0$.
\end{itemize}
To check the first condition, observe that $H^1_0(0,1)$ is an ideal of $H^1(0,1)$, and that the continuity of the values of $f\in \left(H^1(0,1)\right)^m$ in the nodes is not affected by taking the absolute value of $f$.
Let now $f\in V_0$ and $g\in V$. If $\vert g\vert\leq \vert f\vert$, then in particular $\vert g_j(\mv_{n})\vert \leq \vert f_j(\mv_{n})\vert=0$ for all $j\in \Gamma(\mv_{n})$, i.e., $g\in V_0$. Finally, since $A$ generates a positive semigroup, $V_0$ is an ideal of itself and this yields that $g\cdot {\rm sign}f\in V_0$.
\end{proof}

\section{The heat equation on spaces of continuous functions}

Also in view of applications, we are now interested to extend the
previous $L^p$-type well-posedness results to a setting where
continuous functions are considered instead.

Consider the part $\tilde{A}_\infty$ of $A$ in the Banach space $\tilde{X}_\infty:=\left(C[0,1]\right)^m$, whose domain is given by
\begin{equation*}
D(\tilde{A}_\infty)=\left\{f\in \left(C^2(0,1)\cap C^1[0,1]\right)^m:
\hspace{-10pt}
\begin{array}{rl}
&\exists d^f\in {\mathbb C}^{n-1}\times \{0\} \hbox{ s.t. }\\
&(\Phi^+)^\top d^f=f(0),\; (\Phi^-)^\top d^f=f(1),\\
&\hbox{ and }\Phi_w^+ f'(0)-\Phi_w^- f'(1)={\mathbb B}d^f\\
\end{array}
\right\}.
\end{equation*}

\begin{lemma}\label{secto}
The operator $\tilde{A}_\infty$ is sectorial on $\tilde{X}_\infty$, and it generates a compact analytic semigroup of angle $\frac{\pi}{2}$. If the matrix $B$ is a diagonal with negative entries, then such a semigroup is also positive and contractive.
\end{lemma}

\begin{proof}
1) Let us first consider the case $B=0$. Then by Theorem~\ref{comp} and Theorem~\ref{anal} we deduce that all the operators $A_{p}$, $p\in[1,\infty]$, are dissipative and sectorial of angle $\frac{\pi}{2}$. In particular, for each $\epsilon\in(0,\frac{\pi}{2})$ there exists $M_\epsilon\geq 1$ such that the estimate
\begin{equation}\label{sect}
\Vert \lambda R(\lambda,A_\infty)\Vert_{{\mathcal L}(X_\infty)}\leq M_\epsilon
\end{equation}
holds for all $\lambda\in \{ \mu\in{\mathbb C}: \vert \hbox{arg } \mu\vert <\pi-\epsilon\}$.

Now observe that if $f\in \tilde{X}_\infty$, then $R(\lambda,A_\infty)f=R(\lambda,A_2)f\in D(A_2)$. But one has $D(A_2)\subset \big(H^2(0,1)\big)^m\hookrightarrow \big(C[0,1]\big)^m$, with compact embedding. Therefore, we see that $R(\lambda,A_\infty)f$ is a continuous function for all $\lambda\in \{ \mu\in{\mathbb C}: \vert \hbox{arg } \mu\vert <\pi-\epsilon\}$. It follows that the analogous of~\eqref{sect} also holds with respect to the norm of $\tilde{X}_\infty$, hence $\tilde{A}_\infty$ is sectorial and it generates an analytic semigroup of angle $\frac{\pi}{2}$.

2) The case of general $B$ can be treated as in the proof of Theorem~\ref{anal}, by means of the theory of boundary perturbation for sectorial operators discussed in~\cite{GK91}.
\end{proof}

The main motivation for considering semigroups on $\left(C[0,1]\right)^m$ comes from applications involving semilinear equations, since we can then effectively apply the theory developed, e.g., in~\cite{Lu95} in order to discuss well-posedness and stability. As an elementary, yet motivating example we mention the following system, related to a Hodgkin--Huxley-model describing the transmission of potential along neurons (see~\cite{Sc02} and references therein).

\begin{prop}\label{exa}
Let $\psi_j\in C^2({\mathbb R})$, $j=1,\ldots,m$. Consider for $t> 0$ the semilinear parabolic network problem given by
\begin{equation*}
\left\{\begin{array}{rcll}
\dot{u}_j(t,x)&=& (c_j u_j')'(t,x)+(\psi_j(u_j(t,x))', &x\in(0,1), \; j=1,\dots,m,\\
u_j(t,\mv _i)&=&u_\ell (t,\mv _i)=: d^u_i, &j,\ell\in \Gamma(\mv _i),\; i=1,\ldots,n,\\
\sum_{h=1}^{n-1} b_{ih} d^u_h&=& \sum_{j=1}^m \phi_{ij} c_j(\mv_i) u'_j(t,\mv _i), & i=1,\ldots,n-1,\\
d^u_n(t)&=&0.
\end{array}
\right.
\end{equation*}
Then for all $u_0\in (C[0,1])^m$ the Cauchy problem associated to such a system admits a unique (global) mild solution $u$ that depends continuously (with respect to the sup-norm) on the initial data. In fact, $u$ satisfies the problem pointwise for $t>0$.
\end{prop}

\begin{proof}
Rewrite the above system as a semilinear abstract Cauchy problem
\begin{equation*}
\left\{\begin{array}{rcll}
\dot{u}(t)&=& \tilde{A}_\infty u(t)+\Psi(u(t)), &t> 0,\\
u(0)&=&u_{0},\\
\end{array}
\right.
\end{equation*}
on $\tilde{X}_\infty$. Here
$\Psi
$ is the Nemitsky operator defined by
$$\Psi(u)(\cdot):=\begin{pmatrix}
\frac{d}{dx}\left(\psi_1(u_1(\cdot))\right)\\
\vdots\\
\frac{d}{dx}\left(\psi_m(u_m(\cdot))\right) \\
\end{pmatrix}.$$
Taking into account Lemma~\ref{secto}, we are in a setting that is analogous to that of~\cite[\S~7.3.3]{Lu95}. Now, mimicking the proof of~\cite[Prop. 7.3.6]{Lu95} the claim follows.
\end{proof}

A thorough treatment of well-posedness and stability of semilinear diffusion problems over networks goes beyond the scope of this paper. We will deal with such an issue in the forthcoming paper~\cite{CM06}.

\bigskip
Even in the linear case (i.e., $\psi_1\equiv 0$, $j=1,\ldots,m$), the problem considered in Proposition~\eqref{exa} is not well-posed in a classical sense. In fact, albeit sectorial (hence the generator of an analytic semigroup), the operator $\tilde{A}_\infty$ is not densely defined in $\tilde{X}_\infty$, thus the generated semigroup is not strongly continuous.

By the theorem of Stone--Weierstrass, the already defined space $C_0(G)$ of continuous function over the network that vanish in $\mv_{n}$)  is the closure of $D(\tilde{A}_\infty)$. 

\begin{theo}
The part $\bf A$ of $\tilde{A}_\infty$ in $C_0(G)$ generates a compact, strongly continuous semigroup which is analytic of angle $\frac{\pi}{2}$. If further $B$ is diagonal with negative entries, then such a semigroup is contractive, real, positive,  and uniformly exponentially stable.
\end{theo}

\begin{proof}
Reasoning as in the proof of Proposition~\ref{exa}, we deduce from Theorem~\ref{comp} that ${{\bf A}}$ is a resolvent positive operator on $C_0(G)$. Since $\bf A$ is also densely defined, by~\cite[Thm. 3.11.9]{ABHN01} it generates a positive strongly continuous semigroup $({\bf T}(t))_{t\geq 0}$.

By Lemma~\ref{secto}, we see that $\bf A$ is sectorial and dissipative: this yields the analyticity (with angle $\frac{\pi}{2}$) and the contractivity of $({\bf T}(t))_{t\geq 0}$. Observe further that the $p$-independence of the spectrum of $A_p$ (by Theorem~\ref{comp}) yields the invertibility of ${\bf A}$, hence the uniform exponential stability of $({\bf T}(t))_{t\geq 0}$.

Finally, in order to show that the semigroup is compact, observe that due to its analyticity $T_2(t)$ maps $X_2$ into $D(A_2^\infty)\subset\left(C^\infty[0,1]\right)^m\cap C_0(G)\subset D({\bf A})$ for all $t>0$. Thus, denoting by $X_{\bf A}$ the Banach space obtained by endowing $D({\bf A})$ with the graph norm, we have
$${\bf T}(t)=i_{X_{\bf A},C_0(G)}\circ T_2(t)\circ i_{C_0(G),X_2},\qquad t>0.$$
Here $i_{C_0(G),X_2}$ and $i_{X_{\bf A},C_0(G)}$ denote the canonical imbeddings of $C_0(G)$ into $X_2$ and of $X_{\bf A}$ into $C_0(G)$, respectively. It follows from the theorem of Ascoli-Arzel\`a that the latter imbedding is compact, so that also ${\bf T}(t)$ is compact for $t>0$, and the claim follows.
\end{proof}

We can finally draw a conclusion that is similar to Theorem~\ref{wellp}, and can be proved likewise.

\begin{theo}
The network diffusion problem is well-posed on $C_0(G)$, i.e., for all initial data $u_0\in  C_0(G)$ it admits a unique classical solution that continuously depends on the initial data. The sup-norm of the solution tends to 0 in time.
\end{theo}

\section{A technical lemma}

The following result seems to be of independent interest. Its proof is due to Wolfgang Arendt, whom we warmly thank.

\begin{lemma}\label{wolfg}
Let $A=(a_{ih})$ be a $n\times n$ matrix with complex-valued coefficients. Then the semigroup $(e^{tA})_{t\geq 0}$ generated by $A$ is $\ell^\infty$-contractive, i.e.,
$$\Vert e^{tA} x\Vert_\infty \leq \Vert x\Vert_\infty,\qquad t\geq 0,\; x\in{\mathbb C}^{n},$$
if and only if
\begin{equation}\label{wolfgcond}
{\rm Re}a_{ii} + \sum_{h\not=i} \vert a_{ih}\vert\leq 0\qquad \hbox{for all }i=1,\ldots,n.
\end{equation}
\end{lemma}

\begin{proof}
The proof goes in two steps.

1) Let us first assume the semigroup $(e^{tA})_{t\geq 0}$ generated by the matrix $A=(a_{ih})$ to be positive, i.e., to have real entries that are positive off-diagonal. Then it is known that $(e^{tA})_{t\geq 0}$ is $\ell^\infty$-contractive if and only if $A{\mathbb 1}\leq 0$, i.e., if and only if
\begin{equation}\label{wolfgcond2}
a_{ii} + \sum_{h\not=i} a_{ih} \leq 0\qquad \hbox{for all }i=1,\ldots,n.
\end{equation}

2) Let us now consider the case of a general matrix $A$ and define a new matrix $A^\sharp=(a^\sharp_{ih})$ by
$$a^\sharp_{ih}:=\left\{
\begin{array}{ll}
{\rm Re}a_{ii}\qquad &\hbox{if } h=i,\\
\vert a_{ih}\vert & \hbox{if } i\not=h.
\end{array}
\right.$$
It is known (see~\cite{De84}) that $A^\sharp$ generates the modulus semigroup of $(e^{tA})_{t\geq 0}$, i.e., the (unique) semigroup $(e^{tA^\sharp})_{t\geq 0}$ that dominates $(e^{tA})_{t\geq 0}$ in the sense of positive semigroups, and is dominated by any other semigroup also dominating $(e^{tA})_{t\geq 0}$.

Let us first assume that~\eqref{wolfgcond} holds. Since
$(e^{tA^\sharp})_{t\geq 0}$ is positive, by part 1) it is also
$\ell^\infty$-contractive. Now, since $(e^{tA^\sharp})_{t\geq 0}$
dominates $(e^{tA})_{t\geq 0}$, it follows that $(e^{tA})_{t\geq
0}$ is $\ell^\infty$-contractive as well.

 Conversely, let
$(e^{tA})_{t\geq 0}$ be $\ell^\infty$-contractive. In order to
show that~\eqref{wolfgcond} holds, consider the modulus semigroup
$(e^{tA^\sharp})_{t\geq 0}$, which is positive and,
by~\cite[Prop.~2.26]{Ou04}, also $\ell^\infty$-contractive. One
can check directly that the adjoint $(e^{tA^{\sharp}*})_{t\geq 0}$
of the modulus semigroup also dominates the adjoint
$(e^{tA*})_{t\geq 0}$, which by duality is $\ell^1$-contractive.
Now it follows from the proof of~\cite[Prop.~2.5]{BG86} that the
semigroup $(e^{tA^{\sharp}*})_{t\geq 0}$ is also
$\ell^1$-contractive, and by duality the positive semigroup
$(e^{tA^\sharp})_{t\geq 0}$ is $\ell^\infty$-contractive. Thus, by
part 1) the entries of $A^\sharp$ satisfy the
condition~\eqref{wolfgcond2}, i.e., \eqref{wolfgcond} holds. This
concludes the proof.
\end{proof}

\medskip

Received February 2006; revised December 2006.

\medskip
\end{document}